\documentclass[11pt,reqno]{amsart}

\usepackage{amsmath,mathtools}
\usepackage{amsthm}
\usepackage{amssymb}
\usepackage{verbatim}
\usepackage[dvipsnames]{xcolor}
\usepackage{hyperref}
\usepackage{mathtools}
\usepackage{enumitem} 
\usepackage{stmaryrd}
\usepackage{soul}

\usepackage{fullpage}
\def\thin{\hspace{.5pt}}

\usepackage[backgroundcolor=none]{todonotes}

\theoremstyle{definition}

\newtheorem{theorem}{Theorem}[section]
\newtheorem*{theorem*}{Theorem}

\newtheorem{example}[theorem]{Example}
\newtheorem{remark}[theorem]{Remark}

\newtheorem{lemma}[theorem]{Lemma}
\newtheorem{prp}[theorem]{Proposition}
\newtheorem{corollary}[theorem]{Corollary}
\newtheorem{lthm}{Theorem}
\newtheorem*{ques}{Question}

\newcommand{\R}{\mathbb{R}}
\newcommand{\Q}{\mathbb{Q}}
\newcommand{\C}{\mathbb{C}}

\newcommand{\Z}{\mathbb{Z}}
\newcommand{\Li}{\operatorname{Li}}

\newcommand{\sign}{\operatorname{sign}}
\tikzset{circ/.style = {fill, circle, inner sep = 0, minimum size = 3}}

\newcommand{\re}{\textnormal{Re}}
\newcommand{\im}{\textnormal{Im}}

\renewcommand{\=}{\;=\;}

\renewcommand{\Re}{\re}
\renewcommand{\Im}{\im}

\usepackage{color}

\newcommand{\Res}{\operatorname{Res}}

\usepackage{commath}
\usepackage{float}
\usepackage{enumitem}
\newcommand{\NN}{\mathbb{N}}
\newcommand{\RR}{\mathbb{R}}
\newcommand{\CC}{\mathbb{C}}
\newcommand{\ZZ}{\mathbb{Z}}

\newcommand{\SSS}{\mathcal{S}}
\newcommand{\VVV}{\mathcal{V}}
\newcommand{\UUU}{\mathcal{U}}
\newcommand{\fbr}[1]{\left\lfloor #1 \right\rfloor}
\newcommand{\br}{\del}
\renewcommand{\mod}[1]{\ (\mathrm{mod}\ #1)}
\tikzset{circ/.style = {fill, circle, inner sep = 0, minimum size = 3}}
\usetikzlibrary{arrows.meta}
\usetikzlibrary{decorations.markings}
\usetikzlibrary{decorations.pathmorphing}
\usetikzlibrary{positioning}
\usetikzlibrary{intersections}
\usepackage{tkz-euclide}
\pgfarrowsdeclarecombine{twolatex'}{twolatex'}{latex'}{latex'}{latex'}{latex'}
\pgfarrowsdeclarecombine{twolatex'}{twolatex'}{latex'}{latex'}{latex'}{latex'}
\tikzset{->-/.style = {decoration={markings,
 mark=at position #1 with {\arrow[scale=2]{latex'}}},
 postaction={decorate}}}
\tikzset{-<-/.style = {decoration={markings,
 mark=at position #1 with {\arrowreversed[scale=2]{latex'}}},
 postaction={decorate}}}

\definecolor{sageblue}{rgb}{0,0,1}
\definecolor{sagered}{rgb}{1,0,0}
\definecolor{sagegreen}{rgb}{0.0, 0.5019607843137255, 0.0}
\definecolor{sageorange}{rgb}{1.0, 0.6470588235294118, 0.0}

\newcommand{\bea}{\begin{equation}\begin{aligned}}
\newcommand{\eea}{\end{aligned}\end{equation}}
\allowdisplaybreaks

\title{On a conjecture of Andrews and almost alternating sign patterns}
\author[Kalita]{Jayashree Kalita}
\address{Department of Mathematics,
	Vanderbilt University, Nashville, Tennessee, United States}
\email{jayashree.kalita@vanderbilt.edu}
\author[Kundu]{Debanjana Kundu}
\address{Department of Mathematics and Statistics, University of Regina, Saskatchewan, Canada}
\email{debanjana.kundu@uregina.ca}
\author[Storzer]{Matthias Storzer}
\address{School of Mathematical Sciences,
	University College Cork, Cork, Ireland}
\email{mstorzer@ucc.ie}
\author[Wang]{Xintong Wang}
\address{Department of Mathematical Sciences, Durham University, Durham, UK}
\email{xintong.wang@durham.ac.uk}
\date{\today}

\begin{document}

\begin{abstract}
In this paper, we prove a sign phenomenon first observed by Andrews for certain $q$-series from Ramanujan's Lost Notebook.  For three of the series considered by Andrews, namely $v_2(q)$, $v_3(q)$, and $v_4(q)$, we show that the coefficients are alternating in sign, with only a density-zero set of exceptions. Our approach yields precise asymptotic formulas for the coefficients via an adapted circle method, inspired by the work of Folsom--Males--Rolen--Storzer on the $q$-series $v_1(q)$, revealing an interplay between exponential growth and oscillatory behaviour. This interaction produces a dominant alternating sign factor, which governs the sign regularity observed numerically by Andrews. More broadly, we establish the same sign behaviour for explicit infinite families of $q$-hypergeometric series encompassing these examples, and show that it arises systematically from oscillatory asymptotics of these $q$-series near roots of unity. We introduce an additional family whose coefficients appear to exhibit similar sign regularity, suggesting that this phenomenon is widespread and may point towards a deeper underlying theory.
\end{abstract}

\maketitle

\section{Introduction}
\label{intro}
The study of $q$-series and their coefficients plays a central role in analytic number theory, dating back to the pioneering work of Hardy--Ramanujan~\cite{hardy1918asymptotic} on the partition function, and its extension to a remarkable exact formula by Rademacher~\cite{rademacher1938partition}.
In many classical settings, the coefficients of such series exhibit regular asymptotic behaviour, often governed by modular or automorphic properties.
In contrast to this classical paradigm, Andrews~\cite{andrews1986questions} noticed in Ramanujan's Lost Notebook~\cite{RlostIV,RlostV} that certain $q$-series exhibit strikingly unusual behaviour.
While the coefficients of many classical partition-theoretic $q$-series tend to infinity or remain bounded, Andrews pointed out five $q$-series whose coefficients grow differently.
His first example was
\begin{equation}\label{eq:sigma}
\sigma(q)
\;\coloneqq\; \sum_{n\geq0}\frac{q^{n(n+1)/2}}{(-q;q)_n} 
\coloneqq\sum_{n\geq 0} S(n)\, q^n,
\end{equation}
where $(a;q)_n \coloneqq \prod_{j=0}^{n-1} (1-aq^j)$ is the $q$-\textit{Pochhammer symbol}.
He observed that the coefficients grow very slowly, but nevertheless appear to be unbounded.
More precisely, he conjectured that 
\begin{equation}\label{conjS}
\limsup_{n\to\infty}\abs{S(n)} = \infty \text{ and } S(n)=0 \text{ for infinitely many }n.
\end{equation}

In their paper, Andrews--Dyson--Hickerson~\cite{andrews1988partitions} proved these conjectures by revealing an unexpected connection between $\sigma(q)$ and the arithmetic of $\mathbb Q(\sqrt 6)$.
Cohen~\cite{cohen1988q} subsequently placed this connection in the setting of Maaß wave forms.
Later developments, including the work of Zwegers~\cite{zwegersmock} 
and Li--Röhrig \cite{li2025mock}
on mock Maaß theta functions 
 and Zagier's~\cite{zagier2010quantum} theory of quantum modular forms, showed that this example was not an isolated curiosity, but an early manifestation of a wider phenomena connecting $q$-hypergeometric series, partitions, and automorphic objects.

In addition to $\sigma$, Andrews discussed four other $q$-series and their coefficients from Ramanujan's Lost Notebook \cite[(1.9)--(1.11)]{RlostIV};
namely,
\begin{align*}
v_1(q)
&
\;\coloneqq\;
 \sum_{n\geq0}\frac{q^{{n(n+1)}/2}}{(-q^2;q^2)_n} 
\= 1
 + q
 + q^6
 - q^7
 - q^8
 + q^9
 + \cdots
\;\eqqcolon\;
\sum_{n\geq 0} V_1(n)\thin q^n,
 \\[5pt]
 v_2(q) 
&
\;\coloneqq\;
 \sum_{n\geq1}\frac{q^{2n^2-n}}{(-q;q^2)_n} 
\= q
 - q^2
 + q^3
 - q^4
 + q^5
 - q^9
 + \cdots
\;\eqqcolon\;
\sum_{n\geq 0} V_2(n)\thin q^n,
\\[5pt]
v_3(q)
&
\;\coloneqq\;
\sum_{n\geq0}\frac{(-1)^nq^{n(n+1)/2}}{(-q;q)^2_n}
\= 1
 - q
 + 2\*q^2
 - 2\*q^3
 + 2\*q^4
 + \cdots
\;\eqqcolon\;
\sum_{n\geq 0} V_3(n)\thin q^n
,
\\[5pt]
v_4(q)
&
\;\coloneqq\;
\sum_{n\geq0}\frac{(-1)^nq^{2n^2}}{(-q;q^2)_{2n}}
\= 1
 - q^2
 + q^3
 - q^4
 + 2\*q^5
 + \cdots
\;\eqqcolon\;
\sum_{n\geq 0} V_4(n)\thin q^n.
\end{align*}

The coefficients $V_1(n)$, $V_2(n)$, $V_3(n)$, and $V_4(n)$ of these $q$-series have interpretations in terms of partitions and overpartitions.
For $V_1(n)$ and $V_4(n)$, these are given in \cite[Definitions~2 and~3]{RlostIV}, while for $V_2(n)$ and $V_3(n)$ they are provided in Section~\ref{sec:partitions}.
We note that upon performing more extensive SAGE computations during the course of this project, we realized that the interpretation for $V_4(n)$ requires minor modification, which we have addressed in Section~\ref{sec:partitions}.

A particularly interesting phenomenon about these $q$-series, as pointed out by Andrews, is that their coefficients ``appear to have great sign regularity''.
To illustrate this behaviour, the coefficients $V_2(n)$ for $0\leq n\leq 200$ are plotted in Figure~\ref{fig:coeffV2} (coloured by the parity of $n$) and the coefficients~$V_2(175),\ldots,V_2(185)$ are listed in Table~\ref{table:V2n}.

\begin{figure}[h]\label{fig:coeffV2}
\centering
\includegraphics[width=.6\textwidth]{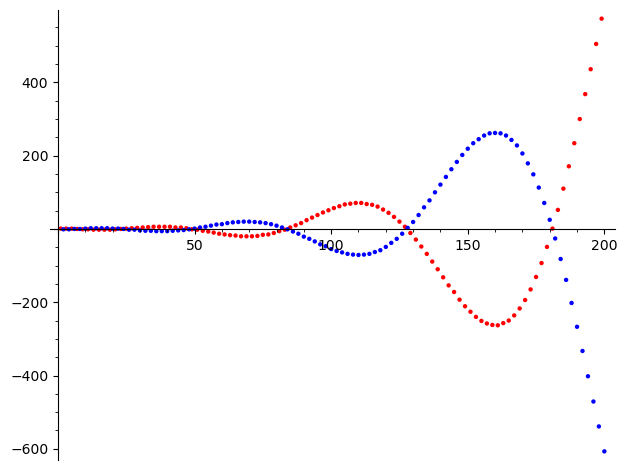}
\caption{The coefficients $V_2(n)$ for $0\leq n\leq 200$.}
\end{figure}

\begin{table}[H]
\label{table:V2n}
\begin{tabular}
{c | c c c c c c c c c c c c c c c c c c c c}
 $n$ & 175 & 176 & 177 & 178 & 179 & 180 & 181 & 182 & 183 & 184 & 185
 \\\hline
 $V_2(n)$ &$-131$ & $+113$ & $-93$ & $+71$ & $-49$ & $+25$ & $+1$ & $-26$ & $+52$ & $-82$ & $+110$
\end{tabular}
\caption{The coefficients~$V_2(175),\ldots,V_2(185)$}
\end{table}

Andrews made precise conjectures about the coefficients $V_1(n)$ and suggested that analogues should be true for $V_{2}(n),V_{3}(n)$, and $V_4(n)$.
These conjectures have now been resolved for $V_1(n)$ in \cite{Bac26,folsom2023oscillating}.
More precisely, \cite{folsom2023oscillating} proves Conjecture 4 and a density-one version of Conjecture 3 of Andrews, while \cite{Bac26} proves Andrews’ Conjectures 5 and 6.
This paper establishes the analogues of \cite[Conjectures 3 and 4]{andrews1986questions} for $V_2(n)$, $V_3(n)$, and $V_4(n)$, and develops a general framework for similar functions.
The main result of our paper is the following theorem 
which confirms the predictions made by Andrews, and is proved in Section~\ref{sec:pfs}.
It would be interesting to extend the results from \cite{Bac26} to the functions $V_2(n)$, $V_3(n)$, and $V_4(n)$.

First, let us recall that a set $A$ is said to have \textit{density} 
$\mathfrak{d}$ if 
\[
\lim_{X \to \infty} \frac{|A \cap \{1, 2, \ldots, X\}|}{X} = \mathfrak{d}.
\]

\begin{lthm}
\label{thm:sgnpattern}
The sequences $V_2(n),V_3(n)$, and $V_{4}(n)$ satisfy an \emph{almost alternating sign pattern}, meaning that
the following assertions are true for $j=2, 3, 4$.
\begin{enumerate}[label=\textup{(}\roman*\textup{)}]
 \item The sequence $\abs{V_j(n)} \to \infty$ as $n\to \infty$ away from a set of density $0$.
 \item For almost all $n$, the coefficients $V_j(n)$ and $V_j(n+1)$ have opposite signs.
\end{enumerate}
\end{lthm}

We will explain Theorem~\ref{thm:sgnpattern} by establishing precise asymptotic formulae for the coefficients, which depend on the parity of $n$ and feature an exponential term together with an oscillating factor.
The following theorem is proved in Section~\ref{sec:circle}.

Recall the \emph{Bloch--Wigner dilogarithm}
$D(e^{i\theta}) \coloneqq \Im(\Li_2(e^{i\theta}))$, and more generally the definition 
\[
D(x) := \Im(\Li_2(x))+\arg (1-x) \log\abs{x}.
\]

\begin{lthm}
\label{thm:B}
For $j\in \{2,3,4\}$, as $n\to\infty$, the following asymptotics hold
\begin{equation}
\label{eq:Viasymp}
\begin{split}
V_j(n)\!\!\!
\= (-1)^{n}\alpha_j\frac{e^{2\sqrt{n}\Re\left(\sqrt{W_j}\right)}}{\sqrt{n}}
\cos\left(2\sqrt{n}\Im\bigl(\sqrt{W_j}\bigr)
+ \frac{\pi}{4}\right)
\!\left(1+O\left(n^{-\frac{1}{2}}\right)\right)
+ O\Biggl(\frac{e^{\sqrt{n}\Re\bigl(\sqrt{W_j}\bigr)}}{\sqrt{n}}\Biggr), 
\end{split}
\end{equation}
where
\begin{align*}
&\alpha_2 \= \frac 1{\sqrt{2}\sqrt[4]{3}},\qquad\qquad\qquad\quad
&&W_2 \= -\frac{\pi^2}{24}+D(e^{\pi i/3})\frac i2,\\
&\alpha_3 \= -\frac{1}{\sqrt[4]{3}},\qquad\qquad
&&W_3 \= D(e^{\pi i/3})\frac i2,\\
&\alpha_4 \= \frac1{2\sqrt[4]{3}},\qquad\qquad
&&W_4 \= D(e^{\pi i/3})\frac i2,
\end{align*}
and $D(e^{\pi i/3})=1.014942\cdots$.
\end{lthm}

More generally, our results in Sections~\ref{sec:circle} and~\ref{sec:pfs}
show that the phenomena observed by Andrews are not isolated.
In particular, we study the asymptotics and prove an almost alternating sign pattern for the coefficients of a class of $q$-hypergeometric series exhibiting special behaviour as $q$ approaches a root of unity; see Theorems~\ref{Fourier: mainasym} and \ref{thm:gen}.
We provide explicit examples of infinite families (beyond Andrews's $q$-series) where our results hold; see Corollaries~\ref{cor: asympmain} and \ref{cor: asympU}.

The functions studied in this paper are closely related to $v_1(q)$, but they belong to a different analytic family.
Our methods apply to $q$-series whose summands have a generic shape like
\[
\frac{(-1)^{(b+1)n}q^{2n^2+bn}}{\br{-q; q^2}_{rn}} \qquad \text{ with } b\in\ZZ \text{ and } r\in\{1,2\},
\]
whereas $v_1(q)$ involves the denominator $(-q^2;q^2)_n$.
Our results address the remaining Andrews series $v_2(q)$, $v_3(q)$, and $v_4(q)$, together with a broader family to which they naturally belong.

To illustrate the asymptotics of the coefficients $V_2(n)$, the values $(-1)^n \sqrt{n}e^{-2\sqrt{n}\Re(\sqrt{W_2})}V_2(n)$ for $0\leq n\leq 500$ are plotted in Figure~\ref{fig:scaledV2}.
Similar observations can be made for the coefficients $V_3(n)$ and $V_4(n)$.
\begin{figure}[H]\centering
\includegraphics[width=0.6\textwidth]{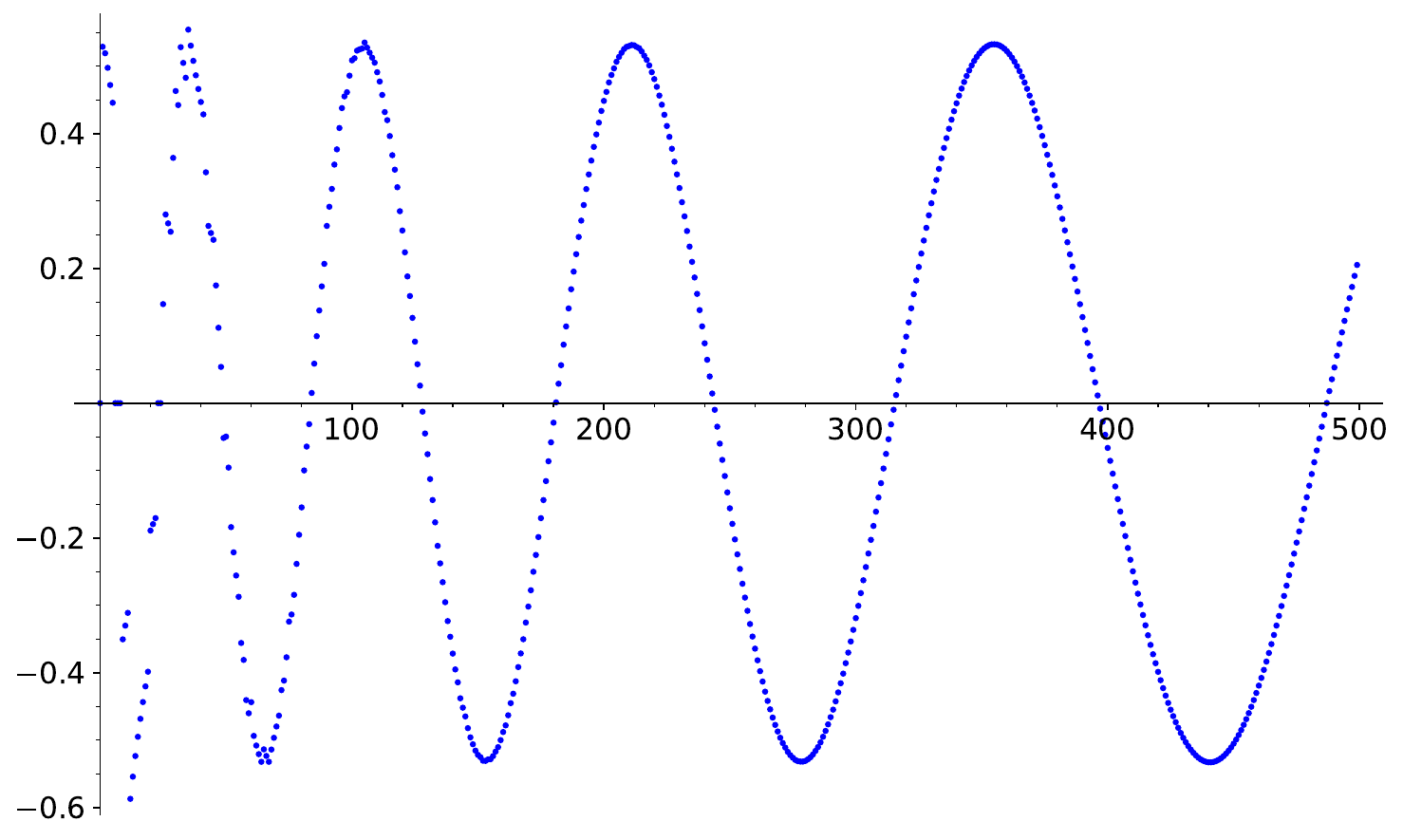}
\caption{The values $(-1)^{n}\sqrt{n}e^{-2\sqrt{n}\Re(\sqrt{W_2})} V_2(n)$ for $0\leq n\leq 500$.}
\label{fig:scaledV2}
\end{figure}

Let $j\in \{2,3,4\}$.
Following the philosophy of the circle method, the asymptotics of the coefficients $V_j(n)$ can be obtained from the asymptotics of the generating function $v_j(q)$ for $q$ near a root of unity.
For example, when $j=2$, the main contribution for $v_2(q)$ comes from $q$ near the dominant singularity at $-1$ and we have
\bea
\label{eq:asympv2zeta1}
v_2(-e^{-z})
\=\left[\sqrt{\frac{\pi}{4\sqrt{3} z}}
\biggl(e^{W_2/z}(1+i) +e^{\overline{W}_2/z}(1-i)
\biggr)
-1\right]
\thin (1+O(z)),
\eea
as $z\to 0$ on a ray in the right half-plane.
Applying the circle method, we see that this causes the factor $(-1)^n$ in~\eqref{eq:Viasymp} producing the \textit{dominant alternating sign factor}.
The oscillation of the corresponding generating functions $v_j(q)$ gives rise to an oscillation in the coefficients in~\eqref{eq:Viasymp}.
Together, this implies the almost alternating sign pattern with exceptions for the pattern near the zeros of the cosine factor in Theorem~\ref{thm:B}.
More generally, we show in Theorem~\ref{thm:gen} that coefficients with an asymptotic behaviour as in~\eqref{eq:Viasymp} satisfy the sign pattern described in Theorem~\ref{thm:sgnpattern}.

With an outlook towards future work, we define a new family of $q$-series for $k\in\Z_{\geq2}$,
\[
v^{\cbr{k}}(q)
\coloneqq 
\br{q;q}_\infty\sum_{n\geq 1}\frac{q^{n(n-1)}}{(-q;q^k)_n}\eqqcolon
\sum_{n\geq 0} V^{\cbr{k}}(n)q^n.
\]
Our numerical experiments led us to question if the following sequence of length $k$,
\[
V^{\{k\}}(n),V^{\{k\}}(n+1), \ldots,V^{\{k\}}(n+k-1)
\]
satisfies a sign regularity for almost all values of $n$; this is made precise in Section~\ref{sec:future}.
An ambitious goal is to reformulate these kinds of questions (and results) within an automorphic framework.

The rest of the paper is organized as follows.
In the same spirit as Andrews \cite{RlostIV} who provides a partition-theoretic interpretation for $V_1(n)$, we give interpretations of the coefficients $V_2(n)$, $V_3(n)$, and $V_4(n)$ in terms of partitions and overpartitions in Section~\ref{sec:partitions}.
Section~\ref{sec:preliminaries} is preliminary in nature; we recall relevant background and give asymptotic results.
In Section~\ref{sec:asympgenf}, we prove asymptotics of $q$-series generalizing $v_2(q)$ and $v_4(q)$ as $q$ approaches a root of unity.
The asymptotics for $v_3(q)$ follows from an identity that relates $v_1$, $v_3$, and $v_4$; see~\eqref{eq:pointed by KB}.
Next, we perform the circle method to obtain asymptotics for the coefficients of these $q$-series in Section~\ref{sec:circle}.
We then show in Section~\ref{sec:pfs} that any sequence with asymptotics of the form \eqref{eq:Viasymp} satisfies the sign pattern described in Theorem~\ref{thm:sgnpattern}; see Theorem~\ref{thm:gen}.
In Section~\ref{sec:future}, we introduce a new
family of $q$-series whose coefficients seem to follow a similar sign pattern.

\section*{Acknowledgements} 
We are grateful to Koustav Banerjee for pointing out equations~\eqref{eq:pointed by KB} and~\eqref{eq:pointed by KB2}, and to Joshua Males for helpful discussions during the initial stages of this project.
We thank Kathrin Bringmann, Herbert Gangl, Robert Osburn, and Larry Rolen for their encouragement and continued support, as well as George Andrews and Peter Sarnak for their interest in this work.
This project was initiated when DK was supported by a PIMS Postdoctoral Fellowship.
She acknowledges the support of an AMS-Simons Early Career Travel Grant (2024-25) and NSERC Discovery Grant RGPIN-2026-07384 and DGECR-2026-00254.
She is also grateful for the research environment provided by Lodha Mathematical Science Institute (LMSI) Mumbai where she was in residence during Fall 2025.
During the preparation of this article, MS was partially funded by the Max-Planck-Society and the Irish Research Council Advanced Laureate Award IRCLA/2023/1934.

\section{Combinatorial interpretations}
\label{sec:partitions}

In this section, we discuss the interpretation of the coefficients $V_2(n),V_3(n),$ and $V_4(n)$ in terms of partitions and overpartitions.
It would be beneficial to give a combinatorial explanation of the almost alternating sign pattern of $V_2(n),V_3(n),$ and $V_4(n)$.

An \textit{(integer) partition} of a non-negative integer $n$ is a way of writing $n$ as a sum of positive integers.
As explained in \cite[Definition~3]{RlostIV},
the coefficients $V_4(n)$ count the difference between the number of partitions of $n$ into an even number of parts and an odd number of parts such that the partitions satisfy the following conditions:
\begin{enumerate}[label=(\alph*)]
 \item no part is divisible by $4$,
 \item at least one even part,
 \item each integer congruent to $2 \mod{4}$ and at most equal to the largest part in the partition occurs exactly once,
 \item the largest part is incongruent to $1\mod4$.
\end{enumerate}
We emphasize that Andrews's interpretation of $V_4(n)$ included the first three conditions \textit{only}.
This coincides with data when $n\leq 6$.
However, further calculations show that condition (d) also needs to be included in the partition interpretation of $V_4(n)$.

We begin this section by reviewing some basic notions.
Recall from \cite{corteel2004overpartitions} that an \emph{overpartition} of a non-negative integer $n$ is a non-increasing sequence of natural numbers that sum to $n$, where the first occurrence of any part \textit{may be} overlined.
For example, the overpartitions of $3$ are $(3)$, $(\overline{3})$, $(2,1)$, $(\overline{2},1)$, $(2,\overline{1})$, $(\overline{2},\overline{1})$, $(1,1,1)$, $(\overline{1},1,1)$.
The total number of overpartitions of $n$ is denoted by $\overline{p}(n)$ and is given via the generating function
\[
\sum_{n\geq 0} \overline{p}(n) q^n \= \prod_{n\geq 1} \frac{1+q^n}{1-q^n} \ = 1
 + 2\*q
 + 4\*q^2
 + 8\*q^3
 + 14\*q^4
 + \cdots \,.
\]

\subsection{Overpartition interpretation for \texorpdfstring{$V_2$}{}}

Consider overpartitions of $n$ for which there exists an integer 
$K \ge 1$ such that:
\begin{enumerate}[label = (\alph*)]
 \item Each positive integer of the form $4k-3$ with $1 \le k \le K$ (that is, each integer in $\{1,5,9,\dots,4K-3\}$) appears exactly once as an \emph{overlined} part.
 \item For each odd integer
 \[
 j \in \{1,3,\dots,2K-1\},
 \]
 additional occurrences of $j$ are permitted as non-overlined parts.
 \item No parts other than those specified in (a) and (b) are allowed to appear.
 \end{enumerate}
Let $P_{+}(n)$ (resp. $P_-(n)$) count the number of overpartitions satisfying (a), (b), and (c)
with an even (resp. odd) number of unmarked parts.

We have
\begin{align*}
\sum_{n\geq0}(P_+(n)-P_-(n))q^n
&=\sum_{K\geq1}
q^{1+5+9+\cdots+(4K-3)}
\sum_{\substack{m_1,\dots,m_{K}\ge 0}}
(-1)^{m_1+\cdots+m_{K}}
q^{m_1\cdot 1+m_2\cdot 3+\cdots+m_k(2K-1)}\\
&=\sum_{K\geq 1}\frac{q^{1+5+9+\cdots+ (4{K}-3)}}{(1+q)(1+q^3)\cdots (1+q^{2{K}-1})}
=\sum_{K\geq 1}\frac{q^{2K^2-K}}{
(-q;q^2)_{K}}=v_2(q).
\end{align*}

\begin{example}
{Suppose that $n=3$.
Only the partition $(\overline{1},1,1)$ contributes to $P_+(3)$ with $K=1$.
None of the other overpartitions contribute to either $P_+(3)$ or $P_-(3)$.}
\end{example}

\begin{example}
Now suppose that $n=7$.
We list all the \textit{fifteen} partitions of $7$: $(7)$, $(6+1)$, $(5+2)$, $(5+1+1)$, $(4+3)$, $(4+2+1)$, $(4+1+1+1)$, $(3+3+1)$, $(3+2+2)$, $(3+2+1+1)$, $(3+1+1+1+1)$, $(2+2+2+1)$, $(2+2+1+1+1)$, $(2+1+1+1+1+1)$, $(1+1+1+1+1+1+1)$.
The only values of $K$ that can be considered are $1$ or $2$.
When $K=1$, the only overpartition that can contribute is $(\bar{1}+1+1+1+1+1+1)$ and it contributes to $P_+(7)$. On the other hand, when $K=2$ we have $(\bar{1}+\bar{5}+1)$ which contributes to $P_-(7)$.
We see that $P_+(7)=P_-(7)=1$ and hence the coefficient of $q^7$ in $v_2(q)$ is $V_2(7) = 0$.
\end{example}
 
\subsection{Overpartition interpretation for \texorpdfstring{$V_3$}{}}

Consider marked overpartitions
of $n$ where markings are specified in the following way: \begin{enumerate}[label = (\alph*)]
 \item Each integer
 \[
 1,2,\dots,k
 \]
 (where $k\ge 1$ is the largest part) appears once and only once as an \emph{unmarked} part.

 \item For each $j\in\{1,2,\dots,k\}$, additional marked occurrences of $j$ are allowed.
 Each additional occurrence may be marked in one of two distinct ways; that is, each such copy may carry one of two overpartition-type decorations:
 \[
 \text{Type A marking},\qquad \text{Type B marking}.
 \]
 Parts of the same size but with different markings are regarded as distinct, and all required occurrences from (a) remain unmarked.
 The order of the markings does not matter.

 \item No parts other than those specified in (a) and (b) are allowed to appear.
 \end{enumerate}
Let $Q_{+}(n)$ (resp. $Q_-(n)$) count the number of overpartitions satisfying (a), (b), and (c)
with an even (resp. odd) number of parts.

Then
\begin{align*}
\sum_{n\geq0}\br{Q_+(n)-Q_-(n)}q^n
&=\sum_{k\geq 1} (-1)^k q^{1+2+\cdots+k}
 \sum_{m_1,\dots,m_k\ge 0}\sum_{n_1,\dots,n_k\ge 0}
 (-1)^{\sum_{j=1}^k(m_j+n_j)}
 q^{\sum_{j=1}^k j(m_j+n_j)}\\
&=\sum_{k\geq1}\frac{(-1)^k q^{1+2+3+\cdots+k}}{(1+q)^2 (1+q^2)^2 \cdots (1+q^k)^2}\\
&=\sum_{k\geq1}\frac{(-1)^k q^{k(k+1)/2}}{(-q;q)_k^2}=v_3(q).
\end{align*}

\begin{example}
Suppose that $n=3$.
The partition $(3)$ does not satisfy (a) and hence does not contribute to either $Q_{+}(3)$ or $Q_-(3)$.
Next we note that $(2,1)$ contributes to $Q_+(3)$ since it has two (even) parts.
This partition of $3$ cannot have Type-A or Type-B markings since it would violate (a).
Next we consider the partition $(1, 1, 1)$ of $3$. This does not contribute to $Q_+(3)$ or $Q_-(3)$ by (a).
But we can have Type-A and Type-B markings which contribute to $Q_-(3)$, namely $(1, 1^{A}, 1^{B})$, $(1, 1^{A}, 1^{A})$, and $(1, 1^{B}, 1^{B})$. 
So, $Q_+(3)=1, Q_-(3)=3$ and the coefficient of $q^3$ in $v_3(q)$ is $V_3(3) = -2$.
\end{example}

\section{Preliminaries}
\label{sec:preliminaries}
We begin by introducing the analytic setup required for the asymptotic analysis in later sections.

\subsection{Dilogarithms}\label{sec:polylogs}
We write $q=\zeta e^{-z}$, where $\zeta$ is a root of unity and $z$ lies in the right half-plane.
For our analysis of the asymptotic results as $z\to0$ on any ray in the right half-plane, it is necessary to work with non-standard branches of the logarithm suitable for our setting.

To this end, we fix $\varphi \in \mathbb{C}$ with $\abs{\varphi}=1$ and $\abs{\arg \varphi}<\frac{\pi}{2}$, and choose a branch of the logarithm, denoted by $\widetilde{\log}$, such that the function
\[
\Li_1^\varphi(e^{-iv}) \coloneqq -\widetilde{\log}(1-e^{-iv})
\]
has branch cuts precisely along the lines 
\[
\mathcal D\coloneqq \{v\in\C \;|\; \Re((v+2\pi n)/\varphi)=0, \ \Im(v)>0, n\in\Z\}.
\]
This corresponds to rotating the standard branch cuts by the angle $\arg(\varphi)$; see Figure~\ref{fig:li}.

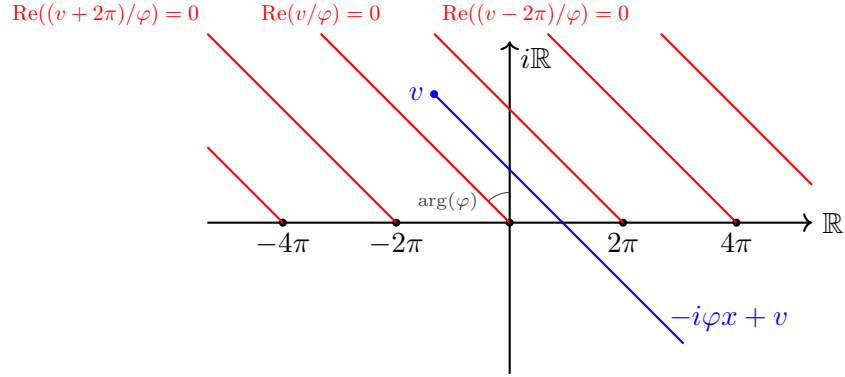
\begin{figure}[H]
\centering
 \begin{tikzpicture}
 \draw [->, thick] (-4, 0) -- (4, 0) node [pos=1.07, left] {$\RR$};
 \draw [->, thick] (0, -2) -- (0, 2.4) node [pos=0.95, right] {$i\RR$};

 \node [circ] at (0, 0) {};
 \node [circ] at (-1.5, 0) {};
 \node [circ] at (-3, 0) {};

 \node [circ] at (1.5, 0) {};
 \node [circ] at (3, 0) {};
 \node [blue, circ] at (-1, 1.7) {};

 \draw [red, thick] (0, 0) -- (-2.5, 2.5) node[scale=0.75] [pos=1, above] {$\Re(v/\varphi)=0$};
 \draw [red, thick] (-1.5, 0) -- (-4, 2.5) node[scale=0.75] [pos=1, above left] {$\Re((v+2\pi)/\varphi)=0$};
 \draw [red, thick] (-3, 0) -- (-4, 1);
 \draw [red, thick] (1.5, 0) -- (-1, 2.5) node[scale=0.75] [pos=1, above right ] {$\Re((v-2\pi)/\varphi)=0$};
 \draw [red, thick] (3, 0) -- (0.5, 2.5);
 \draw [red, thick] (2, 2.5) -- (4, 0.5);
 \draw [blue, thick] (-1, 1.7) -- (2.3, -1.6) node [pos=0.9, right] {$-i\varphi x+v$};
 \draw [darkgray] (0, 0.4) arc(90:135:0.4) node [left] {\tiny{$\arg(\varphi)$}};

 \node [blue, left] at (-1, 1.7) {$v$};
 \node [below] at (-1.5, 0) {$-2\pi$};
 \node [below] at (-3, 0) {$-4\pi$};
 \node [below] at (1.5, 0) {$2\pi$};
 \node [below] at (3, 0) {$4\pi$};
 
\end{tikzpicture}
\caption{The branch cuts of $\Li_s^\varphi(e^{-iv})$, with $s=1, 2$ and $-i\varphi x+v$ for $x\geq 0$.}
 \label{fig:li}
\end{figure}

Following the standard construction of the dilogarithm, we define $\Li_2^\varphi$ for $v\in\CC \setminus \mathcal D$
by
\begin{align}
\label{eq:dilogdef}
\Li_2^\varphi(e^{-iv})
:= -\int_0^{e^{-iv}} \widetilde\log(1-u)\frac{du}{u}
= (-\varphi)\int_0^\infty \widetilde\log\bigl(1-e^{-\varphi x-iv}\bigr)\,dx,
\end{align}
where the path of integration is chosen so that it does not cross any branch cut of $\Li_1^\varphi(e^{-iv})$ as depicted in Figure~\ref{fig:li}.
We refer the reader to \cite{zagier2007dilogarithm} for more on the dilogarithm function.

As in the case of the standard polylogarithm, we define the \textit{polylogarithm} $\Li^{\varphi}_{-m}$ with the new branch cuts inductively by
\[
\Li_{-m}^\varphi(z)\coloneqq z\frac{d}{dz}\Li_{-m+1}^\varphi(z) 
\qquad\qquad (m \in \Z_{\ge0}).
\]
When $m\geq 0$, the function $\Li^{\varphi}_{-m}$ is a rational function; hence it is independent of $\varphi$ and the branch of~$\Li_1^\varphi$.

\subsection{Asymptotic tools}
In this section, we record basic asymptotic results which will be required for the proofs of our main results.
Here and throughout, we use the notation $e(x)\coloneqq e^{2\pi i x}$.

The results in the following lemma are standard, see, e.g., \cite[Lemma~2.1]{folsom2023oscillating}

\begin{lemma}
\label{le:est}
Let $\zeta=e(a/m)$ be a root of unity of order $m\in\NN$ and $a$ be an integer such that $\gcd(a,m)=1$.
Then the following assertions are true.
\begin{enumerate}[label = \textup{(}\arabic*\textup{)}]
 \item
 Let $\alpha\in\RR_{>0}$, $n_0\in\Z$ and for $t\in \CC$,
 set $\pm=\sign(\Re(t))$.
 As $\abs{t}\to\infty$ along a ray in $\CC$ with $\Re(t)\neq 0$,
 \[
 \sin(\alpha(it-n_0))=\mp\frac{1}{2i}\exp(\pm\alpha(t+n_0i)) 
 (1+o(t^{-L})) \text{ for all }L\in\NN.
 \]
 \item 
 Define the Dedekind sum $s(x,y)$ for coprime integers $x,y$ as 
 \bea\label{eq:dedekindsum}
 s(x,y)\coloneqq\sum_{l=1}^{y-1}\frac{l}{y}
\biggl(\frac{xl}{y}-\fbr{\frac{xl}{y}}-\frac{1}{2}\biggr).
 \eea
 As $z\to 0$ in the right half-plane, i.e., as $q=\zeta e^{-z}\to \zeta$,
 \[
 (q;q)_\infty = \exp
 \biggl(-\frac{\pi^2}{6m^2z}+\frac{z}{24}\biggr)
 \sqrt{\frac{2\pi}{mz}}e\biggl(\frac{s(-a,m)}{2}\biggr)(1 + o(z^L)) \text{ for all }L\in \NN.
 \]
 \item 
 Suppose that $m$ is even and consider the function
 \[
 Q(\zeta)=e\biggl(\frac{s(-a,m/2)-s(-a,m)}{2}\biggr).
 \]
We have, as $z\to 0$ in the right half-plane,
 \[
 (-q;q)_\infty=\exp\biggl(-\frac{\pi^2}{6m^2z}+\frac{z}{24}\biggr)Q(\zeta)
 (1+o(z^L)) \text{ for all }L\in\NN.
 \] 
 \end{enumerate}
\end{lemma}

We now derive two additional estimates that will be needed in our analysis.

\begin{lemma}
\label{le:est2}
Using the notation of Lemma~\ref{le:est}, the following asymptotics hold.
\begin{enumerate}
\item 
If $m$ is odd, then as $z\to 0$ in the right half-plane,
 \[
 \br{-q;q}_\infty=\exp\biggl(\frac{\pi^2}{12m^2z}+\frac{z}{24}\biggr)\frac{1}{\sqrt{2}}R(\zeta)
 (1+o(z^L)) \text{ for all }L\in\NN,
 \]
 where
 \[
 R(\zeta)=e\biggl(\frac{s(-2a,m) -s(-a,m)}{2}\biggr).
 \]
 \item Suppose that $m\equiv 2\mod 4$.
 As $z\to 0$ in the right half-plane, for all $L\in\NN$,
 \[
 (-q;q^2)_\infty=\exp\biggl(-\frac{\pi^2}{3m^2z}-\frac{z}{24}\biggr)\sqrt{2}\frac{Q(\zeta)}{R(\zeta^2)}
 (1+o(z^L)).
 \]
\end{enumerate}
\end{lemma}
 
\begin{proof}
The statement in (1) follows from Lemma~\ref{le:est}(2) and an application of the identity 
\[
(-q;q)_\infty=\frac{(q^2;q^2)_\infty}{(q;q)_\infty}.
\]
The claim in (2) follows from (1), Lemma~\ref{le:est}(3), and the identity
\[
(-q;q^2)_\infty=\frac{(-q;q)_\infty}{(-q^2;q^2)_\infty}. \qedhere
\]
\end{proof}

Next, we need a result on the asymptotics of the $q$-Pochhammer symbol.
For non-negative integers $k$, the Bernoulli numbers $B_k$ are a sequence of rational numbers defined by the Taylor series expansion
\[
\frac{x}{e^x -1} = \sum_{k\geq 0} B_k \frac{x^k}{k!}.
\]
The \textit{Bernoulli polynomials} $\boldsymbol{B}_k(X)\in\Q[X]$ have the generating function
\[
\frac{te^{Xt}}{e^t -1} = \sum_{k\geq 0} \boldsymbol{B}_k(X) \frac{t^k}{k!}.
\]
They also admit the explicit expression
\bea\label{eq:bernpol}
\boldsymbol{B}_n(X) = \sum_{k=0}^n\binom{n}{k}B_kX^{n-k}.
\eea
In particular, the $n$-th \textit{Bernoulli number} is $B_{n} = \boldsymbol{B}_n(0)$.

The following lemma is a generalization of a result in \cite[Lemma~2.2]{folsom2023oscillating},
adapted to our setting.
Asymptotic expansions of this form are standard in the analysis of $q$-series and arise for instance, in the study of Nahm sums; see for example \cite{zagier2021asymp}.

\begin{lemma}
\label{le:pohest}
Let $w=e^{-iv}\in\CC$, and assume that $\Re((v+2\pi n)/\varphi)\neq 0$ for all $n\in\ZZ$ if $\Im(v)>0$.
Let $\zeta\in\CC$ be a root of unity of order $m$. Let $j,\,k\in\NN$ with $k\mid m$.
Then, as $z\to 0$ in the right half-plane, i.e., as $q=\zeta e^{-z/m}\to\zeta$, the following equality is true
\begin{small}
\[
(wq^j;q^k)_\infty=\exp\Biggl(-\frac{k}{mz}\Li_2^\varphi(({w}{\zeta^{j}})^{m/k}) - \biggl(\frac{k}{2} + \frac{k-j}{m}\biggr)\Li_1^\varphi(({w}{\zeta^{j}})^{m/k})+ \sum_{t=1}^m\frac{t}{m}\Li_1^\varphi(\zeta^{kt-k+j}w) + \psi_{w,\zeta}(z)\Biggr)
\]
\end{small}
where $\psi_{w, \zeta}(z)\in\CC\llbracket z \rrbracket$ has the following asymptotic expansion as $z\to 0$ for all $N\in\NN$
\[
\psi_{w, \zeta}(z)=-\sum_{s=2}^N\sum_{t=1}^m k^{s-1} \boldsymbol{B}_s\br{1-\frac{t-1+j/k}{m}}\Li_{2-s}^\varphi(\zeta^{kt-k+j}w)\frac{z^{s-1}}{s!} + O(z^N).
\]
\end{lemma}

\begin{proof}
Throughout the proof, set $z=\varphi h$, where $h\in\RR_{>0}$ and $\varphi\in\CC$ with $\abs{\varphi}=1$ and $\abs{\arg(\varphi)}<\frac{\pi}{2}$.
The branch cuts of $\Li^\varphi$ are on the upper half-plane, so we require $\Re((v+2\pi n)/\varphi)\neq 0$ when $\Im(v)>0$.
By definition of the $q$-Pochhammer symbol 
\[
\left(wq^j;q^k\right)_\infty = \prod_{n=1}^\infty \left( 1- wq^j q^{(n-1)k}\right).
\]
Taking the $\log$ on both sides (and choosing the branch of the logarithm explained in Section~\ref{sec:polylogs}), we obtain
\[
\widetilde{\log}\br{wq^j;q^k}_\infty=\sum_{n\geq 1}\widetilde{\log}\br{1-wq^{k(n-1)+j}}=\sum_{t=0}^{m-1}\sum_{l\geq 1}\widetilde{\log}\br{1-\zeta^{-kt-k+j}we^{-\varphi h\br{k(l m-t-1)+j}/m}}.
\]
Under our assumptions, we have $\zeta^{-kt-k+j}we^{-\varphi\br{kx-h\br{kt+k-j}/m}}\neq 1$.
As a result, the function
\[
x\mapsto \widetilde{\log}\br{1-\zeta^{-kt-k+j}we^{-\varphi(kx-h(kt+k-j)/m)}},
\]
and all of its derivatives
\begin{equation}
\label{eq:deriv}
 \frac{d^n}{dx^n}\widetilde{\log}\br{1-\zeta^{-kt
 -k+j}we^{-\varphi(kx-h(kt+k-j)/m)}}=-\br{-k}^n\varphi^n\Li_{1-n}^\varphi\br{\zeta^{-kt-k+j}we^{-\varphi(kx-h(kt+k-j)/m)}}
\end{equation}
are well-defined for all $w=e^{-iv}$.
As $x\to\infty$, the function and its derivatives decay rapidly since $\Li^{\varphi}_{1-N}(z)\in(1-z)^{-N}\CC[z]$.
Applying the Euler--Maclaurin summation formula \cite[p.~13]{zagier2006mellin},
we obtain
\begin{multline}
\label{eq:pohasym}
 \widetilde{\log}\br{wq^j;q^k}_\infty=\sum_{t=0}^{m-1}\frac{1}{h}\int_0^\infty\widetilde{\log}\br{1-\zeta^{-kt-k+j}we^{-\varphi\br{kx-h\br{kt+k-j}/m}}} dx \\
 +\sum_{t=0}^{m-1}\sum_{n=0}^N\frac{(-1)^nB_{n+1}}{(n+1)!}\frac{d^n}{dx^n}\bigg\rvert_{x=0}\widetilde{\log}\br{1-\zeta^{-kt-k+j}we^{-\varphi\br{kx-h\br{kt+k-j}/m}}} h^n +O(h^N)
\end{multline}
for all $N\in\NN$ where the error term only depends on $w, j,k, \varphi$, and $\zeta$, and
we have used
\[
h^N\varphi^{N+1}\int_0^\infty \Li^{\varphi}_{-N}\br{\zeta^{-kt-k+j}we^{-\varphi(kx-h(kt+k-j)/m)}}\frac{\boldsymbol{B}_{N+1}(x/h)}{(N+1)!}dx=O(h^N).
\]

\textit{Simplifying the first summand:}
We calculate the summand appearing in the first term of \eqref{eq:pohasym} and express the integral as a sum.
We have 
\begin{align*}
 \frac{1}{h}\int_0^\infty\widetilde{\log}(1-\zeta^{-kt-k+j}we^{-\varphi(kx-h(kt+k-j)/m)}) dx &= -\frac{1}{kz}\Li_2^\varphi(e^{z(kt+k-j)/m}\zeta^{-kt-k+j}w)\\
 &= -\frac{1}{kz}\sum_{l\geq 0}\frac{\Li_{2-l}^\varphi(\zeta^{-kt-k+j}w)}{l!}\br{\frac{z\br{kt+k-j}}{m}}^l.
\end{align*}

\textit{Rewriting the second summand:}
Evaluating the derivatives in \eqref{eq:deriv} at $x=0$ yields 
\begin{align*}
 &\frac{d^n}{dx^n}\bigg\rvert_{x=0}\widetilde{\log}(1-\zeta^{-kt-k+j}we^{-\varphi\br{kx-h\br{kt+k-j}/m}})\\ &= -\br{-k}^n\varphi^n\Li_{1-n}^\varphi(\zeta^{-kt-k+j}we^{z\br{kt+k-j}/m})\\
 &= -\br{-k}^n\varphi^n\sum_{l\geq 0}\frac{\Li_{1-n-l}^\varphi(\zeta^{-kt-k+j}w)}{l!}\br{\frac{z\br{kt+k-j}}{m}}^l.
\end{align*}
Hence, \eqref{eq:pohasym} can be expressed as
\begin{align*}
 \widetilde\log\br{wq^j;q^k}_\infty = &-\sum_{t=0}^{m-1}\frac{1}{kz}\sum_{l\geq 0}\frac{\Li_{2-l}^\varphi(\zeta^{-kt-k+j}w)}{l!}\br{\frac{z\br{kt+k-j}}{m}}^l\\
 & -\sum_{t=0}^{m-1}\sum_{n=0}^N\frac{k^nB_{n+1}}{(n+1)!}\sum_{l\geq 0}\frac{\Li_{1-n-l}^\varphi(\zeta^{-kt-k+j}w)}{l!}\br{\frac{kt+k-j}{m}}^l z^{l+n}+O(h^N).
\end{align*}
Shifting $n\mapsto n-1$ and summing over $s=n+l=1, \dots, N$, it follows that $\widetilde\log\br{wq^j;q^k}_\infty$ equals
\begin{align*}
 &-\frac{1}{kz}\sum_{t=0}^{m-1}\Li_2^\varphi(\zeta^{-kt-k+j}w)-\sum_{t=0}^{m-1}\sum_{s=1}^N\sum_{n=0}^sk^{n-1}\binom{s}{n}B_n\br{\frac{kt+k-j}{m}}^{s-n}\Li_{2-s}^\varphi(\zeta^{-kt-k+j}w)\frac{z^{s-1}}{s!}\\
 &\quad+O(h^N)\\
 =& -\frac{1}{kz}\sum_{t=0}^{m-1}\Li_2^\varphi(\zeta^{-kt-k+j}w)-\sum_{t=0}^{m-1}\Li_1^\varphi(\zeta^{-kt-k+j}w)\br{\frac{t+1-j/k}{m}-\frac{1}{2}}
 \\
 &\quad-\sum_{t=0}^{m-1}\sum_{s\geq 2} k^{s-1} \boldsymbol{B}_s\br{\frac{t+1-j/k}{m}} \Li_{2-s}^\varphi(\zeta^{-kt-k+j}w)\frac{z^{s-1}}{s!} +O(h^N),
\end{align*}
where we collected all terms with $s\geq N$ in $O(h^N)$ and used~\eqref{eq:bernpol}.
Replacing $t$ by $m-t\in\{1,\dots,m\}$ yields
{\allowdisplaybreaks
\begin{align*}
 & -\frac{1}{kz}\sum_{t=1}^m\Li_2^\varphi(\zeta^{kt-k+j}w) + \sum_{t=1}^m\Li_1^\varphi(\zeta^{kt-k+j}w)\br{\frac{t-1+j/k}{m}-\frac{1}{2}}-\\
 &\quad\sum_{t=1}^{m}\sum_{s\geq 2}k^{s-1}\boldsymbol{B}_s\br{1-\frac{t-1+j/k}{m}} \Li_{2-s}^\varphi(\zeta^{kt-k+j}w)\frac{z^{s-1}}{s!} +O(h^N),
\end{align*}} 
which we simplify and rewrite as
 \[ -\frac{1}{kz}\sum_{t=1}^m\Li_2^\varphi(\zeta^{kt-k+j}w)-\frac{1}{2}\sum_{t=1}^{m}\Li_1^\varphi(\zeta^{kt-k+j}w)+\sum_{t=1}^m\frac{t-1+j/k}{m}\Li_1^\varphi(\zeta^{kt-k+j}w)+\psi_{w,\zeta}(z).
 \]
We now divide the first sum involving $\Li^{\varphi}_2(\cdot)$ into $k$ sums each of length $m/k$ to obtain
\begin{align*}
&-\frac{1}{kz}\br{\sum_{t=1}^{\frac{m}{k}}\Li_2^\varphi\br{(\zeta^k)^t\frac{w}{\zeta^{k-j}}}+\sum_{t=\frac{m}{k}+1}^{2\frac{m}{k}}\Li_2^\varphi\br{(\zeta^k)^t\frac{w}{\zeta^{k-j}}}+\cdots+\sum_{t=(k-1)\frac{m}{k}+1}^{k\frac{m}{k}}\Li_2^\varphi\br{(\zeta^k)^t\frac{w}{\zeta^{k-j}}}}\\
 &\quad -\frac{1}{2}\sum_{t=1}^m\Li_1^\varphi(\zeta^{kt-k+j}w) + \sum_{t=1}^m\br{\frac{t-1+j/k}{m}}\Li_1^\varphi(\zeta^{kt-k+j}w) + \psi_{w,\zeta}(z).
\end{align*}
The ``distribution property'' of the dilogarithm \cite[p.~9]{zagier2007dilogarithm} asserts that for all natural numbers $n$,
\[
\Li^\varphi_2(x) = n \sum_{z^n = x} \Li^\varphi_2(z).
\]
Using this, the first term of the above expression simplifies considerably.
Simultaneously, we rearrange like-terms to obtain
\[
-\frac{k}{mz}\Li_2^\varphi\br{\br{\frac{w}{\zeta^{k-j}}}^{m/k}} - \left( \frac{1}{2} + \frac{1-j/k}{m}\right)\sum_{t=1}^m\Li_1^\varphi(\zeta^{kt-k+j}w) + \sum_{t=1}^m \frac{t}{m}\Li_1^\varphi(\zeta^{kt-k+j}w) + \psi_{w,\zeta}(z).
\]
As before, we rewrite the sum involving $\Li_1^\varphi(\cdot)$ into $k$ sums each of length $m/k$.
Now, recall that (writing $\beta$ to denote a root of unity of order $n$) we have the identity 
\[
\sum_{t=1}^n\Li^{\varphi}_1(\zeta^t w)=\sum_{t=1}^n\widetilde{\log}(1-\beta^t w)=\widetilde{\log}\br{\prod_{t=1}^n\br{1-\beta^t w}}=\widetilde{\log}(1-w^n)=\Li^{\varphi}_1(w^n).
\]
This allows us to simplify the second term of the above expression which we can now rewrite as
{\allowdisplaybreaks
\[
-\frac{k}{mz}\Li_2^\varphi\br{\br{\frac{w}{\zeta^{k-j}}}^{m/k}} - \left(\frac{k}{2} + \frac{k-j}{m}\right)\Li_1^\varphi\br{\br{\frac{w}{\zeta^{k-j}}}^{m/k}} + \sum_{t=1}^m\frac{t}{m}\Li_1^\varphi(\zeta^{kt-k+j}w) + \psi_{w,\zeta}(z).
\]
}
This completes the proof.
\end{proof}

\section{Asymptotics of \texorpdfstring{$v_j(q)$}{}}
\label{sec:asympgenf}
We provide the asymptotics of the functions $v_2(q), v_3(q)$, and $v_{4}(q)$ as $q$ approaches a root of unity.
It is recorded in \cite[(19.7.1)]{RlostV} that

\bea
\label{eq:pointed by KB}
v_3(q) \= v_1(-q)-2v_4(q)+2.
\eea
In particular, the asymptotics for $v_3(q)$ follow directly from that of $v_1(q)$ and $v_4(q)$.
Moreover, the contribution from $v_1(q)$ is contained in an error term from $v_4(q)$.
We remind the reader that the asymptotics of $v_1(q)$ were calculated in \cite[Theorem~1.2]{folsom2023oscillating}.

In order to prove the asymptotics of $v_2(q)$ and $v_4(q)$, we consider the generalization
\begin{equation}
\label{def: qgeneral}
\VVV_{r,b}(q) \coloneq \sum_{n\geq 1}\frac{(-1)^{(b+1) n}q^{2n^2+bn}}{\br{-q; q^2}_{rn}}\eqqcolon \sum_{n\ge0}V_{r,b}(n)q^n ,
\end{equation}
where $b\in\ZZ$ and $r\in\ZZ_{> 0}$.
With this definition, we recover $v_2$ from $\VVV_{r,b}(q)$ when $b=-1, r=1$, and $v_4-1$ by substituting $b=0, r=2$.

The asymptotics of $\VVV_{r,b}(q)$ when $q$ approaches a root of unity $\zeta$ depends on the order $m$ of $\zeta$.
If $m\not\equiv 2 \mod 4$, then $\VVV_{r,b}(q)$ stays bounded and we have the following theorem.

\begin{theorem}
\label{thm:asympv2O1}
If $\zeta$ is a root of unity of order $m\not\equiv2\mod 4$, then as $z\to 0$, we have
$\VVV_{r,b}(\zeta e^{-z}) \to \VVV_{r,b}(\zeta)=O(1)$.
More precisely, we have
\[
\VVV_{r,b}(\zeta)\= 
 \frac 1 {1-4^{-r}}
 \sum_{s=0}^{m-1} (-1)^{(b+1)s}
 \frac{\zeta^{(2s^2+bs)}}{(-\zeta;\zeta^{2})_{rs}}
 \times 
 \begin{cases}
 {(1-(-1)^{b}2^{-r})}&\text{if $m\equiv\pm 1\mod 4$,}\\[8pt]
 1 &\text{if $m\equiv 0\mod 4$.}
 \end{cases}
\]
\end{theorem}

\begin{proof}
Similar to the proof of \cite[Lemma~3.1]{folsom2023oscillating}, we can bound the summands in $\VVV_{r,b}(\zeta e^{-z})$ and use the theorem of dominated convergence to obtain $\VVV_{r,b}(\zeta e^{-z}) \to \VVV_{r,b}(\zeta)$ as $z\to 0$.
We compute
\begin{align*}
 \VVV_{r,b}(\zeta)&=
 \sum_{n\geq1}
 \frac{(-1)^{(b+1)n}\zeta^{2n^2+bn}}
 {(-\zeta;\zeta^{2})_{rn}} &&(n\mapsto s+mn)\\
 &=\sum_{s=0}^{m-1}
 \sum_{n\geq0}
 (-1)^{(b+1)(mn+s)}
 \frac{\zeta^{2(s+mn)^2+b(s+mn)}}
 {(-\zeta;\zeta^{2})_{r(s+mn)}}\\
 &=\sum_{s=0}^{m-1} (-1)^{(b+1)s}
 \frac{\zeta^{2s^2+bs}}{(-\zeta;\zeta^{2})_{rs}}
 \sum_{n\geq 0} \frac{(-1)^{(b+1)mn}}{(-\zeta; \zeta^{2})_{rm}^n}.
 \end{align*}
 One can show that
 \[
 ({-\zeta; \zeta^{2}})_{rm}
 \= 
 \begin{cases}
 2^r &\text{if $m\equiv\pm 1\mod 4$}\\
 4^r &\text{if $m\equiv0\mod 4$},
 \end{cases}
 \]
 and thus, we have
 \begin{align*}
 \VVV_{r,b}(\zeta)
 &=
 \sum_{s=0}^{m-1} (-1)^{(b+1)s}
 \frac{\zeta^{2s^2+bs}}{(-\zeta;\zeta^{2})_{rs}}
 \times 
 \begin{cases}
 \dfrac 1 {1+(-1)^{b}2^{-r}}&\text{if $m\equiv\pm 1\mod 4$,}\\[8pt]
 \dfrac 1 {1-4^{-r}}&\text{if $m\equiv 0\mod 4$.}
 \end{cases}\\
 &=\frac 1 {1-4^{-r}}
 \sum_{s=0}^{m-1} (-1)^{(b+1)s}
 \frac{\zeta^{2s^2+bs}}{(-\zeta;\zeta^{2})_{rs}}
 \times 
 \begin{cases}
 {(1-(-1)^{b}2^{-r})}&\text{if $m\equiv\pm 1\mod 4$,}\\[8pt]
 1 &\text{if $m\equiv 0\mod 4$.}
 \end{cases}
\end{align*}
This completes the proof.
\end{proof}

On the other hand, if $m\equiv2 \mod 4$, then $\VVV_{r,b}(q)$ grows exponentially.

\begin{theorem}
\label{thm:asympv2}
Fix $r\in \{1,2\}$.
Let $\zeta$ be a root of unity of order $m$, and let $b\in \Z$.
If $m\equiv2\mod 4$, write $m=2d$ where $d$ is odd.
Then, as $z\to0$ on a ray in the right half-plane
\[
\VVV_{r,b}(\zeta e^{-z}) \= \left(e^{W_{2r}/(d^2z)}\frac{\gamma_{r,b}(\zeta)}{\sqrt{z}}
 +e^{\overline{W}_{2r}/(d^2z)}\frac{\overline{\gamma}_{r,b}(\zeta)}{\sqrt{z}}
 +\phi_{r,b}(\zeta)\right)
 (1+O(z)).
 \]
 Here, the constants $\gamma_{r,b}(\zeta)$, $\phi_{r,b}(\zeta)\in\C$ are defined in~\eqref{eq:defgammaphi} and 
 \bea\label{eq:Wdef}
 W_{2r} \coloneqq
 \dfrac{D( e^{\pi i/3})i}{2} -\begin{cases}
 \dfrac{\pi^2}{24}&\text{if}\ r=1,\\[8pt]
 0 & \text{if}\ r=2.
 \end{cases}
 \eea
\end{theorem}

Here and throughout, we write $\overline{\gamma}_{r,b}(\zeta)$ and $\overline{W}_{2r}$ for the complex conjugates of ${\gamma}_{r,b}(\zeta)$ and ${W}_{2r}$.
Note that for $r=1,2$, the term $W_{2r}$ agrees with the definitions in Theorem~\ref{thm:B}.

\begin{remark}
In particular, if $\zeta = -1$, then
 \bea\label{eq:asympvvvzeta1}
 \VVV_{r,b}(-e^{-z})\=
 \biggl(
 e^{W_{2r}/z}\frac{\gamma_{r,b}{(-1)}}{\sqrt{z}}
 +e^{\overline{W}_{2r}/z}\frac{\overline{\gamma}_{r,b}{(-1)}}{\sqrt{z}}
 +\phi_{r,b}(-1)\biggr)
 \bigl(1+O(z)\bigr),
 \eea
where
  \begin{equation}
 \label{defgammazeta1}
 \gamma_{r,b}(-1) =
 \sqrt{\frac{\pi}{2r\sqrt{3}}}\exp\left(\frac{2\*r-1-2\*b}{12} \pi i\right)
 \quad \text{and}
 \quad \phi_{r,b}(-1) = -1.
 \end{equation}
If $r=1$ and $b=-1$, \eqref{eq:asympvvvzeta1} recovers the asymptotics of $v_2(-e^{-z})$ in~\eqref{eq:asympv2zeta1}.
\end{remark}

Here, the leading term in the asymptotic depends on $\arg(z)$.
For example, if $\arg(z) = 0$, the constant term $\phi_{r, b}(\zeta)$ is leading.
If $z$ is close to the positive imaginary axis, e.g., $\arg(z)= 0.22 \pi$, the first term $e^{W_{2r}/(d^2z)}\frac{\gamma_{r,b}(\zeta)}{\sqrt{z}}$ is leading.
Similarly, if $z$ is close to the negative imaginary axis, then the second term $e^{\overline{W}_{2r}/(d^2z)}\frac{\overline{\gamma}_{r,b}(\zeta)}{\sqrt{z}}$ is leading.

\subsection{Proof of special case of Theorem~\ref{thm:asympv2}}

Before proving the general case we provide a proof when $\zeta=-1$.
We outline the key steps for the convenience of the reader.
\begin{itemize}
 \item[Step 1:] Write an integral representation for $\VVV_{r,b}(q)$.
 \item[Step 2:] Refine this representation to deduce an asymptotic formula involving analytic functions.
 \item[Step 3:] Use complex analysis to study these analytic functions precisely.
\end{itemize}

We start with Step 1.
Following Watson's contour integral \cite{GR,watson1910continuation}, we have an integral representation for $\VVV_{r,b}(q)$.
For this, we write $L_\infty$ for the contour $L_R$ in Figure~\ref{fig:con} as $R\to\infty$.

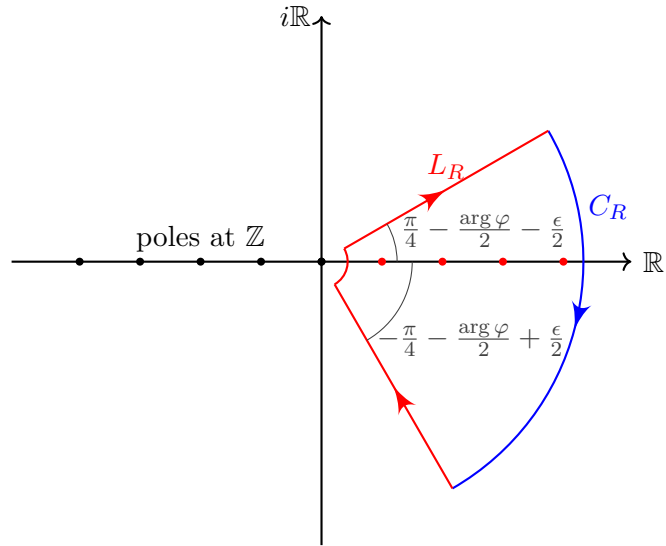
\begin{figure}[H]
\centering
 \begin{tikzpicture}[scale = 1]
 \draw [->, thick] (-4.1, 0) -- (4.1, 0) node [pos=1.07, left] {$\RR$};
 \draw [->, thick] (0, -3.75) -- (0, 3.25) node [pos=1, left] {$i\RR$};

 \node [circ] at (0, 0) {};
 \node [circ] at (-0.8, 0) {};
 \node [circ] at (-1.6, 0) {};
 \node [circ] at (-2.4, 0) {};
 \node [circ] at (-3.2, 0) {};

 \node [circ, red] at (0.8, 0) {};
 \node [circ, red] at (1.6, 0) {};
 \node [circ, red] at (2.4, 0) {};
 \node [circ, red] at (3.2, 0) {};

 \draw [blue, thick, -<-=0.5] (1.732, -3) arc(-60:30:3.464) node [pos=0.8, right ]{$C_R$};
 \draw [red, thick] (0.173, -0.3) arc(-60:30:0.346);
 \node [circ] at (0, 0) {};

 \draw [red, thick, ->-=0.5] (0.3, 0.173) -- (3, 1.732) node [pos=0.5, above] {$L_R$};
 \draw [red, thick, -<-=0.5] (0.173, -0.3) -- (1.732, -3);
 \draw [darkgray] (1.2, 0) arc(0:-60:1.2) node [pos=1, right] {$-\frac{\pi}{4}-\frac{\arg \varphi}{2}+\frac{\epsilon}{2}$};
 \draw [darkgray] (1, 0) arc (0:30:1) node [pos = 0.8, right] {$\frac{\pi}{4}-\frac{\arg \varphi}{2}-\frac{\epsilon}{2}$};

 \node [above] at (-1.6, 0) {poles at $\ZZ$};
 \end{tikzpicture}
 \caption{The contours $L_R$ and $C_R$.}
 \label{fig:con}
\end{figure}
\begin{lemma}
\label{le:integrep}
For $q=-e^{-z}$ with $\Re(z)>0$ and any integer $b$, 
\[
\VVV_{r,b}(q)=\frac{i}{2\br{-q;q^2}_\infty}\int_{L_\infty} e^{-\br{2s^2+bs}z}\br{-e^{-2rsz}q;q^2}_\infty \frac{1}{\sin(\pi s)} ds.
\]
\end{lemma}

\begin{proof}
At $s=n\in\Z$, the function $\frac{1}{\sin(\pi s)}$ has a simple pole with residue $\frac{(-1)^n}{\pi}$.
By applying Cauchy's theorem to the contour illustrated in Figure~\ref{fig:con} for some small $\epsilon>0$, we derive 
\begin{align*}
 &\frac{i}{2\br{-q;q^2}_\infty}\lim_{R\to\infty} \int_{L_R+C_R} e^{-\br{2s^2+bs}z}\br{-e^{-2rsz}q;q^2}_\infty \frac{1}{\sin(\pi s)} ds\\
 = &\frac{\pi}{\br{-q;q^2}_\infty}\sum_{n\geq 1} \underset{s=n}{\Res}\ \br{e^{-\br{2s^2+bs}z}\br{-e^{-2rsz}q;q^2}_\infty \frac{1}{\sin(\pi s)}}\\
 = &\frac{1}{\br{-q;q^2}_\infty}\sum_{n\geq1} (-1)^{n}e^{-(2n^2+bn)z}\br{-q^{2rn+1};q^2}_\infty
 = \VVV_{r,b}(q).
\end{align*}
To complete the proof, it remains to show that the integral over $L_\infty$ converges and that the integral over $C_R$ vanishes as $R\to\infty$.
For details of the proof of these steps, we refer the reader to \cite[Lemma~3.3]{folsom2023oscillating}.
\end{proof}

\subsubsection{A first asymptotic relation}
We continue with the proof of the special case of Theorem~\ref{thm:asympv2}.
Substituting $v=-izs$ in the integral representation of $\VVV_{r,b}(-e^{-z})$ obtained in Lemma~\ref{le:integrep}, we have
\begin{align}
\label{eq:integ2}
 \VVV_{r,b}(q)
 &= \frac{-1}{2z\br{-q;q^2}_\infty}\int_{-izL_\infty} e^{\frac{2v^2}{z}-biv}\br{-e^{-2riv}q;q^2}_\infty \frac{1}{\sin(\pi i v/z)} dv.
\end{align}

\begin{figure}[H]
\centering
\begin{minipage}{0.45\textwidth}
\centering
 \begin{tikzpicture}[scale=1]
 \draw [->, thick] (-3.6, 0) -- (3.6, 0) node [pos=1, anchor=south,yshift=.5mm] {$\RR$};
 \draw [->, thick] (0, -3.6) -- (0, 3.6) node [pos=1, left,xshift=-.5mm] {$i\RR$};

 \node [circ] at (0, 0) {};
 \node [circ, red] at (-0.4, 0.68) {};
 \node [circ, red] at (-0.8, 1.36) {};
 \node [circ, red] at (-1.2, 1.96) {};
 \node [circ, red] at (-1.6, 2.72) {};

 \node [circ, red] at (0.4, -0.68) {};
 \node [circ, red] at (0.8, -1.36) {};
 \node [circ, red] at (1.2, -1.96) {};
 \node [circ, red] at (1.6, -2.72) {};
 
 \draw [blue, thick] (0.3, -0.1732) arc(-30:-120:0.346);
 \node [circ] at (0, 0) {};

 \draw [blue, thick, -<-=0.5] (-0.173, -0.3) -- (-2.08, -3.6);
 \draw [blue, thick, ->-=0.5] (0.3, -0.1732) -- (3.6, -2.0598) node [pos=0.3, right] {$-iz L_\infty$};;
 
 \node [above] at (-2.7, 2) {poles at $iz\ZZ$};
 \end{tikzpicture}
 \end{minipage}\hspace{.05\textwidth}
 \begin{minipage}{.45\textwidth}
 \centering
 \begin{tikzpicture}[scale=1]
 \draw [->, thick] (-3.6, 0) -- (3.6, 0) node [pos=1, anchor=south,yshift=.5mm] {$\RR$};
 \draw [->, thick] (0, -3.6) -- (0, 3.6) node [pos=1, left,xshift=-.5mm] {$i\RR$};

 \node [circ] at (0, 0) {};
 \node [circ, red] at (-0.4, 0.68) {};
 \node [circ, red] at (-0.8, 1.36) {};
 \node [circ, red] at (-1.2, 1.96) {};
 \node [circ, red] at (-1.6, 2.72) {};

 \node [circ, red] at (0.4, -0.68) {};
 \node [circ, red] at (0.8, -1.36) {};
 \node [circ, red] at (1.2, -1.96) {};
 \node [circ, red] at (1.6, -2.72) {};
 
 \node [circ] at (0, 0) {};
 \node [circ] at (-0.45, 0) {};

 \draw[blue, thick, rotate=-150, -<-=0.4] ((0,-1.2) -- (1.6,2.77);
 \draw[blue, thick, rotate=-150, ->-=0.5] ((0,-1.2) -- (-1.6,2.77) node [right] {$\SSS$};

 \node [above] at (-2.7, 2) {poles at $iz\ZZ$};
 \node [above] at (-0.9, 0) {$-\frac{\pi}{6r}$};
 \end{tikzpicture}
 \end{minipage}
 \caption{The contours $-iz L_\infty$ and $\SSS$.}
 \label{figcontourSSS}
\end{figure}
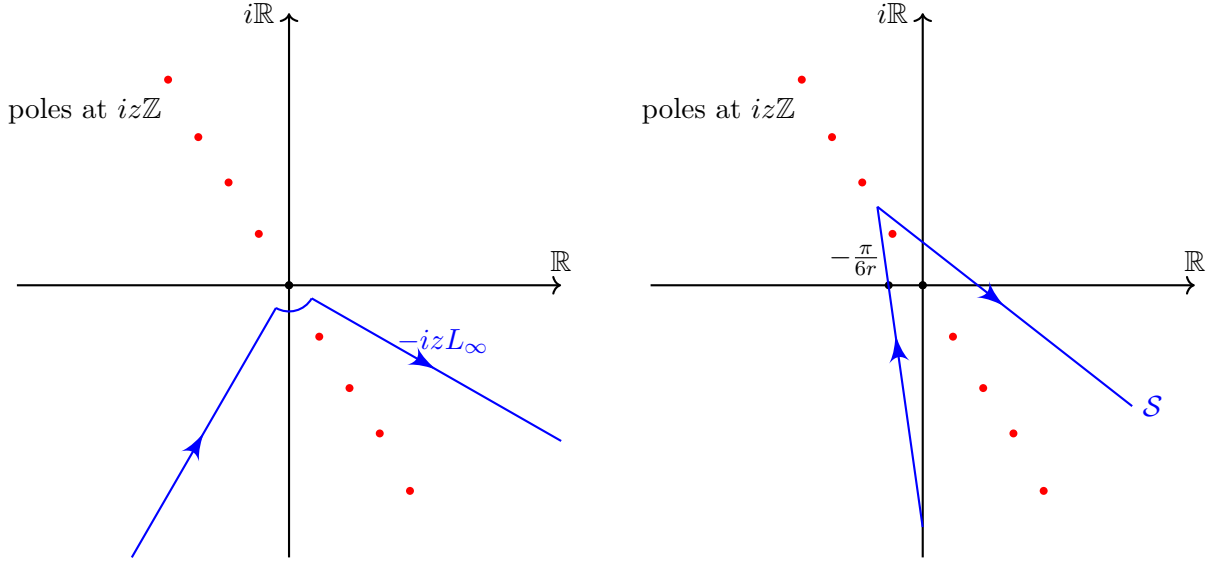

In the integral representation \eqref{eq:integ2} of $\VVV_{r,b}(-e^{-z})$, we modify the contour $-izL_\infty$ to a contour $\SSS$, going through the points $-\pi/(6r)$ and $\pi/(6r)$, cf. Figure~\ref{figcontourSSS}.
Since we have to move the contour through some poles at $iz\Z_{<0}$, the integral in~\eqref{eq:integ2} can be expressed as
\begin{multline*}
 \frac{-1}{2z\br{-q;q^2}_\infty}\int_{\SSS} e^{\frac{2v^2}{z}-biv}\br{-e^{-2riv}q;q^2}_\infty \frac{1}{\sin(\pi i v/z)} dv\\
 -\frac{\pi i}{z\br{-q;q^2}_\infty}\sum_{\substack{n\leq 0\\ \abs{zn}<d_0}} \underset{v=-izn}{\Res}\br{e^{\frac{2v^2}{z}-biv}\br{-e^{-2riv}q;q^2}_\infty \frac{1}{\sin(\pi i v/z)} }
\end{multline*}
for some $d_0>0$, representing the distance from $0$ to the point where the contour intersects the line of poles.
The residue of $\frac{1}{\sin(\pi i v/z)}$ at $v=-izn$ is given by $(-1)^{n+1}iz/\pi$ and hence
\begin{equation*}
 \underset{v=-izn}{\Res}\br{e^{\frac{2v^2}{z}-biv}\br{-e^{-2riv}q;q^2}_\infty \frac{1}{\sin(\pi i v/z)} }=-\frac{(-1)^niz}{\pi} e^{(-2n^2-bn)z}\br{-e^{-2rzn}q;q^2}_\infty.
\end{equation*}
As $z\to 0$, the residues can be consolidated into
\begin{align*}
 &\frac{-\pi i}{z\br{-q;q^2}_\infty}\sum_{n\leq 0} \underset{v=-izn}{\Res}\br{e^{\frac{2v^2}{z}-biv}\br{-e^{-2riv}q;q^2}_\infty \frac{1}{\sin(\pi i v/z)} }\\
 &=\frac{-1}{\br{-q;q^2}_\infty}\sum_{n\leq 0}
 (-1)^n
 e^{-\br{2n^2+bn}z}\br{-e^{-2rnz}q;q^2}_\infty \\
 & =-\sum_{n\leq 0}\frac{(-1)^{(b+1)n} q^{2n^2+bn}}{\br{-q;q^2}_{rn}}
 \in-1+ z\Q\llbracket z\rrbracket
\end{align*}
if $q=-e^{-z}$. 
We set
\[
\phi_{r,b}(-1) \coloneqq \lim_{q\to-1}-\sum_{n\leq 0}\frac{(-1)^{(b+1)n} q^{2n^2+bn}}{\br{-q;q^2}_{rn}} = -1.
\]

Applying Lemmas~\ref{le:est2} and~\ref{le:pohest} to the integrand in \eqref{eq:integ2}, we deduce that
\begin{equation}
\label{eq:v2asymp}
\VVV_{r,b}(-e^{-z})=
\left(\frac{-i}{\sqrt{2}z}e^{\frac{\pi^2}{12z}+\frac{z}{24}} \int_\SSS e^{f_r(v)/z} g_r(v) dv + \phi_{r,b}(-1)\right)(1+O(z)),
\end{equation}
where
$\mathcal{S}$ is the contour given in Figure~\ref{fig:con4}, and
\begin{align*}
f_r(v)&\coloneqq -\frac{\Li_2^\varphi (e^{-2riv})}{2}+2v^2-\sign(\Re(v/\varphi))\pi v \text{ and }\\
g_r(v)&\coloneqq \sign(\Re(v/\varphi))\exp\br{-biv}.
\end{align*}

\subsubsection{Analyzing the functions \texorpdfstring{$f_r(v)$}{} and \texorpdfstring{$g_r(v)$}{}}
Because of our branching of $\Li_2^\varphi$ explained in Section~\ref{sec:polylogs}, both $f_r(v)$ and $g_r(v)$ are analytic functions in $v$ on the domain
\begin{equation}
\label{eq:domain}
\CC\setminus\Biggl(\{\varphi i\RR\} \cup \bigcup_{n\in\ZZ\setminus \{0\}}\left\{\frac{n\pi}r+\varphi i\RR_{>0}\right\}\Biggr).
\end{equation}

More precisely, the branch cuts of $f_r(v)$ are depicted in Figure~\ref{fig:con4}.
\begin{figure}[H]
\centering
 \begin{tikzpicture}[scale=0.7]
 \draw [->, thick] (-4.1, 0) -- (4.1, 0) node [pos=1.07, left] {$\RR$};
 \draw [->, thick] (0, -4.1) -- (0, 4.1) node [pos=1, above] {$i\RR$};

 \draw[red, thick] (-2,3.464) -- (2, -3.464) node[scale=0.8] [right] {$\Re(v/\varphi)=0$};
 \draw[red, thick] (3.4,0) -- (1.4, 3.464) node[scale=0.8] [right] {$\Re((rv-\pi)/\varphi)=0$};
 \draw[red, thick] (-3.4,0) -- (-5.4, 3.464) node[scale=0.8] [left] {$\Re((rv+\pi)/\varphi)=0$};
 
 \node [circ] at (0, 0) {};
 \node [circ] at (-1.08, 0) {};
 \node [circ] at (1.08, 0) {};
 \node[circ] at (-0.6,1.04) {};
 
 \draw[blue, thick, rotate=-150, -<-=0.4] ((0,-1.2) -- (3.5,1.25) node [left] {$\mu$};
 \draw[blue, thick, rotate=-150, ->-=0.6] ((0,-1.2) -- (-2.2,2.8) node [right] {$\mu^\prime$};

 \node [above] at (-1.4, 0) {$-\frac{\pi}{6r}$};
 \node [above] at (1.3, 0) {$\frac{\pi}{6r}$};
 \node [above] at (-0.5, 1.04) {$\xi$};
 \end{tikzpicture}
 \caption{The contour $\SSS = \mu \cup \mu'$ (after applying Lemma \ref{le:est2}).}
 \label{fig:con4}
\end{figure}
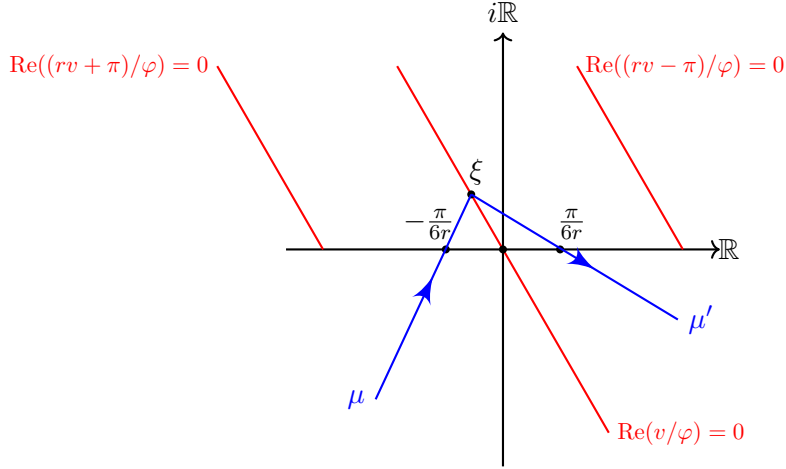

Upon taking derivatives, we see that
\bea
f_r^\prime(v)&=-ri\log(1-e^{-2riv}) + 4v-\sign(\Re(v/\varphi))\pi,\\
f_r^{\prime \prime}(v) &= \frac{2r^2}{e^{2irv}-1}+4.
\eea
The critical points $v_0$ of $f_r$ satisfy $f_r'(v_0) =0$, and thus
$
(1-e^{-2riv_0})^r=-e^{-4iv_0}.
$
Equivalently, $e^{-2riv_0}= e^{\pm\pi i/3}$ if $r=1$ or $r=2$.

Given the shape of the contour and the branch cuts as in
Figure~\ref{fig:con4}, we split the contour
$\SSS=\mu\cup\mu'$ passing through the saddle points
$v_0=-\pi/(6r)$ and $v_0'=\pi/(6r)$.
Let $\xi$ be the intersection point of the ray
$\{\, i\varphi t : t\ge 0 \,\}$ with the line
$v_0 + s(-\sqrt{i\varphi})$, $s\in\mathbb{R}$.
We define the two contour pieces as
\begin{align*}
\mu(s)
&:= v_0+s(\xi - v_0),\qquad{s\in(-\infty,1)}\\
\mu'(s)
&:= v_0'-s(\xi - v_0'),\qquad{s\in(-1,\infty)}
\end{align*}
The contour $\mu$ runs from infinity to $\xi$ passing through
the saddle point $v_0$, while $\mu'$ runs from $\xi$ through
the saddle point $v_0'$ to infinity.

If we parameterize $\mu$ as $v=-\pi/(6r)+is\sqrt{z}$ in a small neighbourhood around $v_0$, the contribution corresponding to the stationary point $v_0=-\pi/(6r)$ becomes 
\bea\label{finalintzeta1}
\sqrt{\frac{1}{2z}}e^{\pi^2/12z+f_r(v_0)/z 
}\int_{-\infty}^\infty\exp\left({-f_r^{\prime\prime}(v_0)\frac{s^2}{2}
}\right)g_r(-\pi/(6r)+is\sqrt{z})ds.
\eea
We have 
\[
\frac{\pi^2}{12}+f_r(v_0)=\frac{\pi^2}{12}-\frac{\Li_2^{\varphi}(e^{-2riv_0})}{2}+2v_0^2+\pi v_0=\overline {W}_{2r}=\begin{cases}
 -\dfrac{\pi^2}{24}-\dfrac{D( e^{\pi i/3})i}{2} &\text{for}\ r=1,\\[8pt]
 -\dfrac{D( e^{\pi i/3})i}{2} &\text{for}\ r=2,
\end{cases}
\] 
where $D(\cdot)$
denotes the Bloch--Wigner dilogarithm.
Moreover, with $f_{1}''(v_0) = 3-\sqrt{-3}$ and 
$f_2''(v_0) =4\sqrt{-3}$, the integral in~\eqref{finalintzeta1} equals
\begin{align*}
 &\sqrt{\frac{1}{2z}}e^{\pi^2/12z+f_r(v_0)/z
 +bi \frac{\pi}{6r}}
\int_{-\infty}^\infty
\exp\br{
 -f_r^{\prime\prime}(v_0)\frac{s^2}{2}
}ds +O(z^{1/2})\\
\=&  e^{\pi^2/12z+f_r(v_0)/z
 +bi \frac{\pi}{6r}}
\sqrt{\frac{\pi}{zf_r^{\prime\prime}(v_0)}}+O(z^{1/2})\\
\= &\frac{e^{\overline{W}_{2r}/z}}{\sqrt{z}}\overline{\gamma}_{r,b}(-1)+ O(z^{1/2})
\end{align*}
where $\gamma_{r,b}(-1)$ is defined in~\eqref{defgammazeta1}.

Upon parameterizing $\mu^\prime$ as $v=\pi/6r-is\sqrt{z}$ in a small neighbourhood around $v_0^\prime$, we obtain
\begin{align*}
&\sqrt{\frac{1}{2z}}e^{\pi^2/12z+f_r(v_0^\prime)/z
}
\int_{-\infty}^{\infty}\exp\br{-f_r^{\prime\prime}(v_0^\prime)\frac{s^2}{2}
}g_r(\pi/6r-is\sqrt{z})ds
\=\frac{e^{W_{2r}/z}}{\sqrt{z}}(\gamma_{r,b}(-1)+O(z)).
\end{align*}

Plugging in the appropriate asymptotics back in~\eqref{eq:v2asymp} proves~\eqref{eq:asympvvvzeta1}.

\subsubsection{Steepest descent principle (saddle-point method)}
A careful reader might ask why we can change the integral from a ``small neighbourhood'' to $(-\infty,\infty)$.
We discuss this in greater rigour in this section.

Using the Taylor expansion at the saddle point $v_0$, we have
\[
f_r(v)=f_r(v_0)+\frac{f_r''(v_0)}{2}(v-v_0)^2+O((v-v_0)^3),
\]
since $f_r'(v_0)=0$.
Now parameterize the contour locally by
\[
v=v_0+is\sqrt{z},
\]
where $s\in\RR$.
Then
\[
(v-v_0)^2=-s^2z,
\]
and therefore
\[
\frac{f_r(v)-f_r(v_0)}{z}
=
-\frac{f_r''(v_0)}{2}s^2+O(s^3\sqrt{z}).
\]
Taking real parts gives
\[
\Re\!\left(\frac{f_r(v)-f_r(v_0)}{z}\right)
=
-\frac{\Re(f_r''(v_0))}{2}s^2+O(s^3 \sqrt{z} ).
\]
We check that
\[
f_r''(v_0)=
\begin{cases}
3-\sqrt{-3} & \text{when } r=1,\\
4\sqrt{-3} & \text{when } r=2.
\end{cases}
\]
We see that $\Re(f_r''(v_0))\ge 0$ in both cases and the contour is chosen so that the quadratic form has negative real part, i.e.,
\[
\Re\!\left(\frac{f_r(v)-f_r(v_0)}{z}\right)\le -cs^2
\]
for some $c>0$.
In the case, $\Re(f_r''(v_0))=0$, decay is not obtained from the real part of $f_r''(v_0)$ in the original coordinate.
Instead, writing $f_r''(v_0)=\abs{f_r''(v_0)}e^{i\theta}$ and rotating the contour in the direction $v-v_0=s\sqrt{z}\,e^{-i\theta/2}$, we obtain
\[
\frac{f_r(v)-f_r(v_0)}{z}
=
-\frac{\abs{f_r''(v_0)}}{2}s^2 + O(s^3\sqrt{z}),
\]
so that exponential decay follows from the steepest descent direction.

Similarly, at the second saddle point $v_0'=\pi/6r$, we use the local parametrization
\[
v=v_0'-is\sqrt{z}, \qquad \text{ where }  s\in\RR.
\]
The global contour passes through $v_0$ and $v_0'$ along opposite edges of the wedge in Figure~\ref{fig:con4}, and hence the local coordinates differ by a sign.
This leads to the parametrizations $v=v_0+is\sqrt{z}$ and $v=v_0'-is\sqrt{z}$.

Once again, using the Taylor expansion,
\[
f_r(v)=f_r(v_0')+\frac{f_r''(v_0')}{2}(v-v_0')^2+O((v-v_0')^3),
\]
we obtain
\[
\frac{f_r(v)-f_r(v_0')}{z}
=
-\frac{f_r''(v_0')}{2}s^2+O(s^3\sqrt{z}).
\]
We can calculate that
\[
f_r''(v_0')=
\begin{cases}
3+\sqrt{-3} & \text{when } r=1,\\
-4\sqrt{-3} & \text{when } r=2.
\end{cases}
\]
So, we note that in both cases the contour is chosen so that
\[
\Re\!\left(\frac{f_r(v)-f_r(v_0')}{z}\right)\le -c s^2
\]
for some $c>0$, ensuring exponential decay away from $v_0'$.
We refer the reader to \cite{olver} for more details on the method of steepest descent.

\subsection{Proof of Theorem~\ref{thm:asympv2}}\label{sec32}
We now give the proof of the asymptotics of $\VVV_{r,b}(q)$ for $q=\zeta e^{-z}$ where $\zeta=e(\frac{a}{2d})$ is a root of unity of order $2d$ with $d$ odd.

We write 
\begin{equation}\label{sumVVV}
\VVV_{r,b}(q)=\sum_{n_0=1}^{d} \VVV_{r,b}^{[n_0]}(q),\qquad\qquad
\VVV_{r,b}^{[n_0]}(q) \coloneq
\sum_{\substack{n\geq 1 \\ n\equiv n_0 \mod{d}}}\frac{\br{-1}^{(b+1)n} q^{2n^2+bn}}{\br{-q;q^2}_{rn}}
\end{equation}
for $n_0\in\cbr{1,\dots,d}$.
When $\zeta=-1$, we have $d=1$ and $n_0=1$.
So, we did not need to decompose the sum above.
We begin with a lemma which allows us to understand the numerator of the summand in \eqref{sumVVV} if $q=\zeta e^{-z}$.

\begin{lemma}
\label{le:zeta}
Let $\zeta=e(\frac{a}{2d})$ be a root of unity of order $2d$ with $d$ odd and let $b \in \mathbb Z$. Fix $n_0\in\{1,\ldots,d\}$ as above.
If $n\equiv n_0\mod{d}$, then
 \[
 (-1)^{(b+1)n}\zeta^{n(2n+b)}=
 (-1)^{(b+1)n_0}\zeta^{n_0(2n_0+b)}
 \, (-1)^{(n-n_0)/d}.
 \]
\end{lemma}

\begin{proof}
For ease of notation, set $\alpha \coloneq \frac{a}{2d}$.
Let $n =n_0+kd$ for some $k\in \Z$.
We compute
\begin{align*}
 \zeta^{n(2n+b)}&=e(\alpha n(2n+b))
 =e\br{\alpha\br{kd+n_0}\br{2\br{kd + n_0} + b}}\\
 &=e\br{\alpha\br{2k^2d^2+4kdn_0+2n_0^2+bkd+bn_0}}.
\end{align*}
We have $e(2\alpha k^2d^2) = 1$ and $e(4\alpha kdn_0) = 1$, hence
\begin{align*}
 \zeta^{n(2n+b)}&=e\Bigl(\alpha n_0(2n_0+b)\Bigr)e\Bigl(ab\frac{n-n_0}{2d}\Bigr)
 =\zeta^{n_0(2n_0+b)}(-1)^{b(n-n_0)/d}.
\end{align*}
The last equality is because $a$ is odd.
Then, we have, because $d$ is odd,
\bea
(-1)^{(b+1)n}\zeta^{n(2n+b)}
= &(-1)^{(b+1)(n-n_0)/d}
  (-1)^{(b+1)n_0}
  (-1)^{b(n-n_0)/d}
  \zeta^{n_0(2n_0+b)}\\
= &(-1)^{(n-n_0)/d}\,
  (-1)^{(b+1)n_0}
  \zeta^{n_0(2n_0+b)}
\eea as claimed.
\end{proof}

This next lemma generalizes Lemma~\ref{le:integrep}.
The calculations are analogous and we only highlight the key differences in the proof.
We do the calculations for each summand $\VVV_{r,b}^{[n_0]}(q)$ and for the final answer, we can sum over $1\leq n_0 \leq d$.

\begin{lemma}
\label{le:geninteg}
Let \(L_\infty\) denote the limiting contour obtained from \(L_R\), depicted in Figure~\ref{fig:con}, as \(R\to\infty\).
Set $q=\zeta e^{-z}$.
Then
\[
\VVV_{r,b}^{[n_0]}(q)=\frac{i\br{-1}^{{{(b+1)}} n_0}\zeta^{2n_0^2+bn_0}}{2d\br{-q;q^2}_\infty}\int_{L_\infty}e^{-z\br{2s^2+bs}}\frac{\br{-\zeta^{2rn_0}e^{-2rzs}q;q^2}_\infty}{\sin\br{\pi(s-n_0)/d}}ds.
\]
\end{lemma}

\begin{proof}
Using Lemma~\ref{le:zeta}, we have
\begin{align*}
\VVV_{r,b}^{[n_0]}(q) & \= \sum_{\substack{n\geq 1 \\ n\equiv n_0 \mod{d}}}\frac{\br{-1}^{{(b+1)n}} q^{2n^2+bn}}{\br{-q;q^2}_{rn}}.
\\ 
& = \frac{
       \br{-1}^{{(b+1)n_0}}
       \zeta^{2n_0^2+bn_0}
    }{\br{-q;q^2}_\infty}\sum_{\substack{n\geq 1 \\ n\equiv n_0 \mod{d}}} (-1)^{(n-n_0)/d}
    e^{-z(2n^2+bn)}\br{-\zeta^{2rn_0}e^{-2rzn}q;q^2}_\infty.
\end{align*}

Observe that the function $\frac{1}{\sin(\pi(s-n_0)/d)}$ has poles at integral values of $s=n\in\Z$ satisfying $s\equiv n_0\mod{d}$ with residues $\frac{d}{\pi}(-1)^{(n-n_0)/d}$.
By Cauchy's residue theorem, we write
\begin{align*}
&\VVV_{r,b}^{[n_0]}(q)\\
&=-\frac{(-1)^{{(b+1)n_0}}
       \zeta^{2n_0^2+bn_0}}{2\pi i\br{-q;q^2}_\infty}\lim_{R\to\infty}\int_{L_R+C_R}e^{-z(2s^2+bs)}\br{-\zeta^{2rn_0}e^{-2rzs}q;q^2}_\infty\frac{\pi}{d\sin(\pi(s-n_0)/d)}ds
\\
&=\frac{i\br{-1}^{{{(b+1)}} n_0}\zeta^{2n_0^2+bn_0}}{2d\br{-q;q^2}_\infty}\int_{L_\infty}e^{-z(2s^2+bs)}\frac{\br{-\zeta^{2rn_0}e^{-2rzs}q;q^2}_\infty}{\sin(\pi(s-n_0)/d)}ds.
\end{align*}
The proof of convergence is analogous to Lemma~\ref{le:integrep}.
\end{proof}

We now prove a slightly stronger version of Theorem~\ref{thm:asympv2}, which in particular establishes our intended result.

\begin{prp}
\label{prp:asympv2}
Let $\zeta =e\br{\frac{a}{2d}}$ be a root of unity of order $2d$ with $d$ odd.
For $n_0\in\cbr{1,\dots, d}$
\[
\VVV_{r,b}^{[n_0]}(\zeta e^{-z})\=
\Biggl(e^{W_{2r}/(d^2z)}
\frac{\gamma_{r,b}^{[n_0]}(\zeta)}{\sqrt{z}}
+e^{\overline{W}_{2r}/(d^2z)}
\frac{\overline{\gamma}_{r,b}^{[n_0]}(\zeta)}{\sqrt{z}}
+\phi_{r,b}^{[n_0]}(\zeta)\Biggr)
(1+O(z)).
\]
as $q=\zeta e^{-z}\to\zeta$, with $W_{2r}$ as in~\eqref{eq:Wdef} and where $\gamma_{r,b}^{[n_0]}(\zeta), \phi_{r,b}^{[n_0]}(\zeta) \in\CC$ are defined in \eqref{eq:gamm} and~\eqref{eq:phi}.
\end{prp}

\begin{proof}
Set $z=\varphi h$, where $h\in\R_{>0}$ and $\varphi\in\CC$ with $\abs{\varphi}=1$ and $0\neq\abs{\arg(\varphi)}<\frac{\pi}{2}$.
Substituting $s=iv/z$ into the integral representation from Lemma~\ref{le:geninteg}, we obtain
\begin{equation}
\label{eq: to use}
\VVV_{r,b}^{[n_0]}(q)=
-\frac{(-1)^{{(b+1)n_0}}\zeta^{2n_0^2+bn_0}}{2dz(-q;q^2)_\infty}
\int_{-izL_\infty}e^{2v^2/z-biv}
\frac{(-\zeta^{2rn_0}e^{-2riv}q;q^2)_\infty}{\sin(\pi(iv/z-n_0)/d)}
dv.
\end{equation}
We change the contour of integration to a contour $\SSS$ similar to the one in Figure~\ref{figcontourSSS} to make it pass through the points $\pm \frac{\pi}{6rd}$. Since we have to move the contour over the poles at $2izn$ with $n\in\ZZ_{>0}$ and $n\equiv n_0 \mod d$, we obtain the additional term from the residues given by the power series
\[
\sum_{\substack{n\leq 0 \\ n\equiv n_0 \mod{d}}}\frac{\br{-1}^{(b+1)n} q^{2n^2+bn}}{\br{-q;q^2}_{rn}} \in\Q[[z]].
\]
As $q\to\zeta$, we write \bea\label{eq:phi}\phi_{r,b}^{[n_0]}(\zeta) \coloneqq
\sum_{\substack{n\leq 0 \\ n\equiv n_0 \mod{d}}}\frac{\br{-1}^{(b+1)n} \zeta^{2n^2+bn}}{\br{-\zeta;\zeta^2}_{rn}}.
\eea
Consider the following holomorphic function in the domain defined in~\eqref{eq:domain}
\[
f_r(v)\coloneqq- \frac{\Li_2^\varphi\br{e^{-2driv}}}{2d^2}+2v^2-\frac{\sign(v/\varphi)\pi v}{d}.
\]
This function is holomorphic because $\Li_2^\varphi(e^{-2driv})$ undergoes a jump of $4d\pi rv$ when $v$ crosses the branch cut at $\Re(v/\varphi)=0$.
Next define
\begin{multline*}
g_r(v) \coloneqq\sign(\Re(v/\varphi))\\\times \exp\left(iv-\sign(\Re(v/\varphi))\frac{\pi in_0}{d}
-\frac {4d+1} {4d}\Li_1^\varphi(e^{-2div})
+\sum_{t=1}^{2d}\frac{t}{2d}\Li_1^\varphi (-\zeta^{2t+2n_0-1}e^{-2iv}) \right).
\end{multline*}

Note that as $z\to 0$, an application of the asymptotics from Lemmas~\ref{le:est2} and~\ref{le:pohest} to the integrand in~\eqref{eq: to use}, yields that $\VVV_{r,b}^{[n_0]}(q)$ agrees with
\bea\label{SaddlePM}
\frac{i(-1)^{{(b+1)n_0}}\zeta^{2n_0^2+bn_0}R\br{\zeta^2}}{\sqrt 2 d zQ(\zeta)} \exp\biggl(\frac{\pi^2}{12d^2 z}\biggr)
\int_\SSS e^{f_r(v)/z} g_r(v) dv + \phi_{(\alpha)}^{[n_0]}(z)
\eea
up to $O(\sqrt{z})$.
We remind the reader that 
\begin{align*}
Q(\zeta)=e\br{\frac{s(-a,d)-s(-a,2d)}{2}}, \qquad\qquad
R(\zeta^{2}) =e\br{\frac{s(-4a,2d) -s(-2a,2d)}{2}},
\end{align*}
where $s(\cdot, \cdot)$ denotes the Dedekind sum~\eqref{eq:dedekindsum}.

We may now apply the saddle-point method to the integral in~\eqref{SaddlePM}. 
In order to do so, we must find the stationary points of $f_r$.
For this, note that
\begin{align*}
f_r^\prime(v)&\=-\frac{ri\log(1-e^{-2driv})}{d}+4v-\frac{\sign(v/\varphi) 2\pi}{2d},\\
f_r''(v) &\= \frac{2 r^2}{1 - e^{2d r i v}} + 4.
\end{align*}
A stationary point $v_0$ of $f_r$ must satisfy $f_r^\prime(v_0) =0$.
Upon simplification, this implies $(1-e^{-2driv_0})^r=-e^{-4div_0}$ which is equivalent to $e^{-2driv_0}= e^{\pm\pi i/3}$ because $r=1,2.$ In view of the contour $\mathcal S$, we choose $v_0=-\frac{\pi}{6dr}$ and $v_0'=\frac{\pi}{6dr}$.

The contribution to $\VVV_{r,b}^{[n_0]}(q)$ from the saddle point $v_0=-\frac{\pi}{6rd}$ yields
\begin{equation}
\label{star}
-\frac{\zeta^{2n_0^2+bn_0}R\br{\zeta^2}}{\sqrt{2z}dQ(\zeta)}e^{\pi^2/12d^2z+f_r(v_0)/z}
\int_{-\infty}^\infty\exp\br{-f_r^{\prime\prime}(v_0)\frac{s^2}{2}
}
g_r\br{-\frac{\pi}{6rd}+is\sqrt{z}}ds;
\end{equation}
the path of integration passes through a small neighbourhood of $v_0$, and $f_{1}''(v_0) = 3-\sqrt{-3}$ and $f_2''(v_0) =4\sqrt{-3}$.
Note that we have
\[
\frac{\pi^2}{12d^2}+f_r(v_0) \= \frac{\overline{W}_{2r}}{d^2},
\]
where $W_{2r}$ is defined in~\eqref{eq:Wdef} and
\bea
\label{eq:gamm}
\gamma^{[n_0]}_{r,b}(\zeta) \=
&-\frac{\zeta^{2n_0^2+bn_0}R\br{\zeta^2}}{dQ(\zeta)}
g_r\br{-\frac{\pi}{3rm}}
\sqrt{\frac {\pi}{f_r''(v_0)}}.
\eea
Then,~\eqref{star} simplifies to $e^{\overline{W}_{2r}/(d^2z)}{\overline{\gamma}_{r,b}^{[n_0]}(\zeta)}/{\sqrt{z}}$.

Using the other stationary point
$v_0^\prime=\frac{\pi}{6rd}$, we obtain a similar expression as above; the only difference being $v_0$ gets replaced by $v'_0$ and $i$ by $-i$.
In other words, the contribution of the stationary point $v'_0$ is
$e^{W_{2r}/(d^2z)}{\gamma_{r,b}^{[n_0]}(\zeta)}/{\sqrt{z}}$.

Adding the contribution from $v_0$ and $v'_0$ as well as $\phi_{r,b}^{[n_0]}(\zeta)$ proves the claim in the theorem.
\end{proof}

\begin{proof}[Proof of Theorem~\ref{thm:asympv2}]
By \eqref{sumVVV}, the asymptotics of $\VVV_{r,b}(\zeta e^{-z})$ in Theorem~\ref{thm:asympv2} follows directly from Proposition~\ref{prp:asympv2} by summing up the asymptotic contributions from $\VVV_{r,b}^{[n_0]}(\zeta e^{-z})$ for $n_0=1,\ldots,d$ and setting
\begin{equation}
\label{eq:defgammaphi}
\gamma_{r,b}(\zeta) \;\coloneqq \sum_{n_0=1}^d\gamma_{r,b}^{[n_0]}(\zeta),\qquad\qquad
\phi_{r,b}(\zeta) \;\coloneqq \sum_{n_0=1}^d \phi_{r,b}^{[n_0]}(\zeta). \qedhere
\end{equation}
\end{proof}

\begin{corollary}\label{cor:v2v3v4}
Let $\zeta=-1$.
Then, as $z\to 0$ on a ray in the right half-plane,
\begin{align*}
v_2(-e^{-z}) &\= \left[\sqrt{\frac{\pi}{4\sqrt{3} z}}
\biggl(e^{W_2/z}(1+i) +e^{\overline{W}_2/z}(1-i)
\biggr)
-1\right]
\left(1+O\left(z\right)\right), \\
v_4(-e^{-z}) &\= \left[ \sqrt{\frac{\pi}{8\sqrt{3} z}}\left( e^{W_4/z}(1-i) + e^{\overline{W}_4/z}(1+i)\right) \right]\left(1+O(z)\right), \\
v_3(-e^{-z}) &\= v_1(e^{-z}) -2(v_4(-e^{-z})-1)\\
& \= -\left[ \sqrt{\frac{\pi}{2\sqrt{3} z}}\left( e^{W_4/z}(1-i) + e^{\overline{W}_4/z}(1+i)\right) \right]\left(1+O(z)\right).
\end{align*}
\end{corollary}

\begin{proof}
Note that the expression for $v_2(-e^{-z})$ and $v_4(-e^{-z})$ follows directly from Theorem~\ref{thm:asympv2}. 
To obtain the expression for $v_3(-e^{-z})$, we use the relationship \eqref{eq:pointed by KB} and recall from \cite[Theorem~1.2(2)]{folsom2023oscillating} that $v_1(e^{-z}) = O(1)$ as $z\to 0$ in the right half plane.
\end{proof}

\section{Asymptotics of the coefficients \texorpdfstring{$V_{j}(n)$}{}}
\label{sec:circle}

In this section, we establish a general result on the asymptotics of the coefficients of a $q$-hypergeometric series with prescribed asymptotic behaviour near roots of unity; see Theorem~\ref{Fourier: mainasym}.
We then explain how this result recovers Theorem~\ref{thm:B}.
The idea is to use the philosophy of Wright's circle method, which is particularly well-suited to non-modular functions.

Consider a $q$-hypergeometric series
\[
v(q) \coloneq \sum_{n \ge 0}V(n)q^n
\]
such that the following asymptotic condition holds in cones across each root of unity $\zeta$ of order $m$.
As $z \to 0$ along a ray in the right half-plane, $v$ satisfies the following asymptotics 
\begin{equation}
\label{eq: asymp v}
 v(\zeta e^{-z}) = \begin{cases}
 O(1) & \text{ if } m\not\equiv 2\mod{4}\\
 \br{e^{W/(d^2z)}\frac{\gamma(\zeta)}{\sqrt{z}}+e^{\overline{W}/(d^2z)}\frac{\overline{\gamma}(\zeta)}{\sqrt{z}}+\phi(\zeta)}\br{1+O(z)} & \text{ if } m\equiv 2\mod{4}.
 \end{cases}
\end{equation}
Here $d=m/2$ and $W$, $\gamma(\zeta)$, $\phi(\zeta)$ are complex numbers with $W \notin (\infty, 0]$.
As before, $\overline{W}$ and $\overline{\gamma}(\zeta)$ denote the conjugates of $W$ and $\gamma(\zeta)$ respectively.

The asymptotics for the coefficients of this $q$-series are given by the following theorem.

\begin{theorem}
\label{Fourier: mainasym}
Let $v(q)$ and $V(n)$ be as above.
As $n \to \infty$, the following equality is true
\[
V(n) = (-1)^{n}\abs{\gamma}\frac{e^{2\sqrt{n}\Re(\sqrt{W})}}{\sqrt{n \pi}}\cos(2\sqrt{n}\Im(\sqrt{W})+\arg(\gamma))\br{1+O\br{n^{-\frac{1}{2}}}} + O\br{\frac{e^{\sqrt{n}\,\re{\sqrt{W}}}}{\sqrt{n}}},
\]
where $\gamma\coloneq\gamma(-1)$.
\end{theorem}

\subsubsection*{Outline of the proof}
By Cauchy’s integral formula, we express $V(n)$ as a contour integral over a circle $C$ of suitably chosen radius inside the unit disk; see Figure~\ref{fig 7}.
We then decompose the contour into major and minor arcs.
The major arc, labelled as $C'$, is around the dominant pole at~$-1$.
We do not have any information on the modularity of $v(q)$, however, we have the asymptotics of $v(q)$ near roots of unity as in~\eqref{eq: asymp v}.
We use these asymptotics to evaluate the integral over $C'$, which contributes the main term $M(n)$.
The remaining portion of the contour forms the minor arc.
We bound the integral over the minor arc by the contribution from the 6-th order roots of unity, multiplied by the length of the minor arc, which is at most $2\pi$.
This yields an error term $E(n)$, which is almost always exponentially smaller than that of the main term; see Section~\ref{sec:pfs}.
Finally, we apply the saddle point method to simplify the resulting expressions and extract the stated asymptotics.

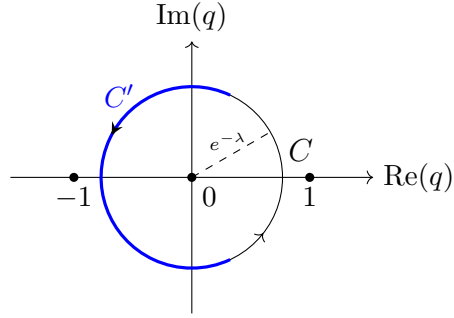
\begin{figure}[H]
	\centering
		\begin{tikzpicture}[scale=1.2]
		\hspace*{-20pt}
			\draw[->] (-2,0) -- (2,0) node[right] {$\Re (q)$};
			\draw[->] (0,-1.5) -- (0,1.5) node[above] {$\Im(q)$};

			\draw (0,0) circle (1);

			\draw[very thick, blue] (65:1) arc (65:295:1);
			\node at (-0.8,0.9) {\textcolor{blue}{$C'$}};

			\node at (1.2,0.3) {\large $C$};

			\draw[->, >=stealth, line width=1pt] (150:1) arc (150:152:1);
			\draw[-{>[scale=1]}] (320:1) arc (320:322:1);

			\draw[dashed] (0,0) -- (30:1);
			\node at (0.4,0.4) {\tiny{$e^{-\lambda}$}};

			\fill (-1.3,0) circle (0.5mm) node[below] {$-1$};
			\fill (1.3,0) circle (0.5mm) node[below] {$1$};
			\fill (0,0) circle (0.5mm) node[below right] {$0$};
		\end{tikzpicture}
 \caption{The contour $C$ and the major arc $C'$ (where $\lambda > 0$).}
 \label{fig 7}
	\end{figure} 

We now turn to the details of the proof.
\begin{proof}[Proof of Theorem~\ref{Fourier: mainasym}]
Using Cauchy's integral formula, we have
\[
V(n)=\frac{1}{2\pi i}\int_C \frac{v(q)}{q^{n+1}} dq,
\]
where $C$ is traversed exactly once in the counter-clockwise direction as seen in Figure~\ref{fig 7}.

As usual, we partition the contour $C$ into major and minor arcs and split the integral accordingly into two parts: 
\begin{equation}
\label{eq: split}
\int_C=\int_{C'}+\int_{C\setminus C'}.
\end{equation}

\subsection*{Major arc estimate}
In this section, we evaluate the first integral in \eqref{eq: split}, which yields the asymptotic contribution from the major arc.
Consider
\[
M(n)=\frac{1}{2\pi i}\int_{C'} \frac{v(q)}{q^{n+1}} dq.
\]
We take the contour \(C\) to be the circle of radius \(e^{-\lambda}\), where \(\lambda \coloneqq \frac{\re{\sqrt{W}}}{\sqrt{n}}\).
The major arc \(C'\) is then parameterized by
\[
q = -e^{-\lambda + i\theta}, \qquad \theta \in [-\delta,\delta],
\]
for some parameter \(\delta > 0\).

We now make the change of variables \(q = -e^{-z}\), so that \(z\) runs from \(\lambda + i\delta\) to \(\lambda - i\delta\), obtaining
\begin{equation} 
\label{eq:cir1}
M(n)=\frac{(-1)^{n}}{2\pi i}\int_{\lambda-i\delta}^{\lambda+i\delta}\frac{v(-e^{-z})}{e^{-nz}} dz.
\end{equation} 
The asymptotics for the generating function is given by \eqref{eq: asymp v}
\begin{equation}
\label{eq:cirasy}
 v(-e^{-z}) = \br{e^{W/z}\frac{\gamma}{\sqrt{z}}+e^{\overline{W}/z}\frac{\overline{\gamma}}{\sqrt{z}}+\phi(-1)}\br{1+O(z)}.
\end{equation}

We now provide a rigorous argument to obtain the contribution from the first exponential term of the above asymptotics.
The contribution of the expression $e^{W/z}\frac{\gamma}{\sqrt{z}}$ to \eqref{eq:cir1} is
\[
\frac{(-1)^{n}\gamma}{2\pi i}\int_{\lambda-i\delta}^{\lambda+i\delta} e^{\tfrac W z+nz}z^{-\frac{1}{2}} dz,
\]
which upon taking $\delta = \sqrt{\tfrac{\abs{W}}{n}}$ and applying the change of variables $z \mapsto z/\sqrt{n}$ becomes
\[
\frac{(-1)^{n}\gamma}{2n^{1/4}\pi i}\int_{\re{\sqrt{W}}-i\sqrt{\abs{W}}}^{\re{\sqrt{W}}+i\sqrt{\abs{W}}}e^{\sqrt{n}(\tfrac W z+z)}z^{-\frac{1}{2}} dz.
\]
We now simplify this integral, which is well-suited for an application of the saddle-point method.
To this end, define
\[
g(z)\coloneq \frac{W}{z}+z.
\]
It is easy to check that $\sqrt{W}$ is a saddle point of $g$.
Indeed, by direct computation 
\[
g'(\sqrt{W})=0.
\]

We now deform the contour so that it passes through this saddle point.
Denote the new contour by $\Gamma$ (cf.~Figure~\ref{fig:Gamma'}).
Then the integral can be written as
\begin{equation}
\label{eq:cir3}
\frac{(-1)^{n}\gamma}{2n^{1/4}\pi i}\int_{\Gamma} e^{\sqrt{n}g(z)}z^{-\frac{1}{2}} dz.
\end{equation}

With the change of variables 
$z=\sqrt{W}+iwn^{-1/4}$, 
the contour $\Gamma$ is mapped to a contour $\Gamma^\prime$ passing through the origin.
We choose $\Gamma$ so that the corresponding contour $\Gamma'$ has a horizontal tangent at the origin (cf.~Figure~\ref{fig:Gamma'}).

\begin{figure}[H]
\begin{tikzpicture}[scale=1.5, >=stealth]

\begin{scope}

\draw[thick, ->] (-2,0) -- (2,0);
\draw[thick, ->] (0,-2) -- (0,2);

\node at (-1.5,1.3) {$z$-plane};

\draw[dashed,->, thick] (0,0) -- (1.3,1.7);
\draw[dashed,->, thick] (0,0) -- (1.3,-1.7);

\fill[blue] (1,1.3) circle (1pt);
\fill[blue] (1,-1.3) circle (1pt);
\fill (0,1.3) circle (1pt);
\fill (0,-1.3) circle (1pt);
\fill (1,0) circle (1pt);

\node[left, font=\scriptsize] at (0,1.3) {$\sqrt{\abs{W}}$};
\node[left, font=\scriptsize] at (0,-1.3) {$-\sqrt{\abs{W}}$};
\node[above, font=\scriptsize] at (1.4,0) {$\re{\sqrt{W}}$};

\draw[blue, thick] (1,1.3) -- (1,-0.4);
\draw[blue, ->, thick] (1,-1.3) -- (1,-0.3);

\node[left, font=\small] at (1,-0.3) {\textcolor{blue}{$\Gamma$}};

\fill[blue] (1,-0.8) circle (1pt);
\node[right, font=\scriptsize] at (1,-0.8) {\textcolor{blue}{$\sqrt{W}$}};
\fill (1,0.8) circle (1pt);
\node[right, font=\scriptsize] at (1,0.8) {$\sqrt{\overline W}$};

\end{scope}
\hspace{-1cm}
\draw[->, bend left=20] (3.0,0.5) to (4.5,0.5);
\node at (3.75,0.9) {$z \mapsto w$};

\begin{scope}[shift={(6.5,0)}]

\draw[->, thick] (-1.9,0) -- (2.9,0);
\draw[->, thick] (0,-2) -- (0,2);

\node at (1,1.3) {$w$-plane};

\draw[blue, ultra thick] (0.9,0) -- (2.4,0);
\draw[blue,->, ultra thick] (-0.8,0) -- (1,0);
\fill[blue] (-0.8,0) circle (1pt);
\fill[blue] (2.4,0) circle (1pt);
\fill[blue] (0,0) circle (1pt);
\node[below, font=\tiny] at (-1.05,0) {$-n^{1/4}\br{\sqrt{\abs{W}} + \im{\sqrt{W}}}$};
\node[below, font=\tiny] at (2,0) {$n^{1/4}\br{\sqrt{\abs{W}} - \im{\sqrt{W}}}$};

\node[above] at (1,0) {\textcolor{blue}{$\Gamma'$}};
\node[above] at (0.2,0) {\textcolor{blue}{$0$}};

\end{scope}

\end{tikzpicture}
\caption{The contours $\Gamma$ and $\Gamma'$.}
\label{fig:Gamma'}
\end{figure}
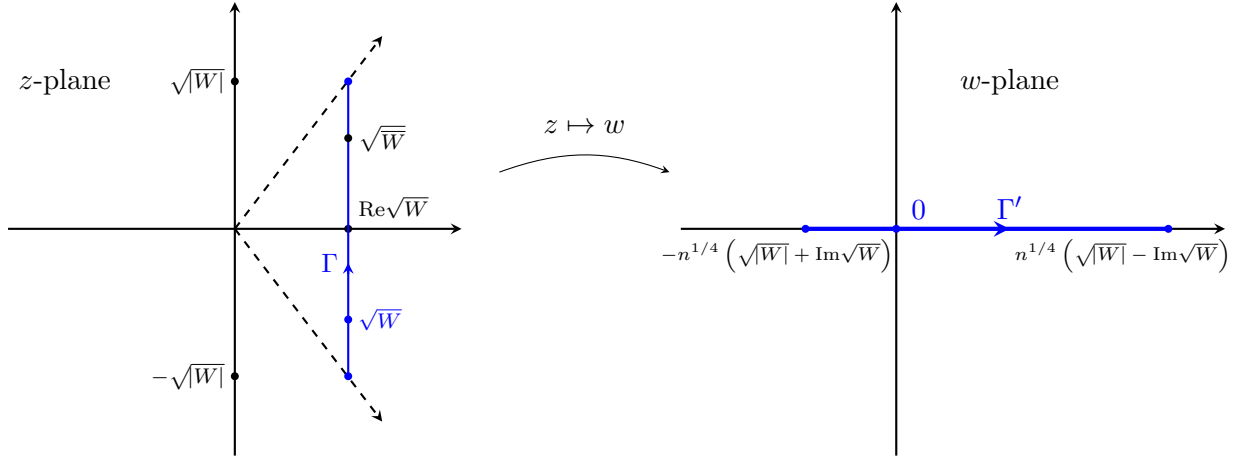

We now expand the functions in the integrand about the saddle point $\sqrt{W}$, expressing them as functions of $w$.
We have
\begin{align*}
\sqrt{n}g(z)&=\sqrt{n}\sum_{r=0}^\infty \frac{g^{\br{r}}(\sqrt{W})}{r!}\br{iw}^r n^{-\frac{r}{4}}\\
&=2\sqrt{nW}-W^{-\frac{1}{2}}w^2+\sum_{r=3}^\infty \br{-i}^r W^{\frac{1}{2}(1-r)} w^r n^{\frac{2-r}{4}},
\end{align*}
which implies
\begin{equation}
\label{eq:cir4}
e^{\sqrt{n}g(z)}= e^{2\sqrt{nW}}e^{-W^{-\frac{1}{2}}w^2}\br{1+O\br{\sum_{r=1}^\infty n^{-\frac{r}{4}}\widetilde{p}_r(w)}},
\end{equation}
where each $\widetilde{p}_r(w)\in\CC[w]$.
Similarly, we have
\begin{equation}
\label{eq:cir5}
z^{-\frac{1}{2}}=W^{-\frac{1}{4}}\br{1+\sum_{r=1}^\infty \frac{\br{-\frac{iw}{2}}^r \br{2r-1}!!}{r!} W^{-\frac{r}{2}}n^{-\frac{r}{4}}}.
\end{equation}
Therefore, substituting \eqref{eq:cir4} and \eqref{eq:cir5} into \eqref{eq:cir3} yields
\[
\frac{(-1)^{n}\gamma}{2\sqrt{n}\pi} e^{2\sqrt{nW}} W^{-\frac{1}{4}} \int_{\Gamma^\prime} e^{-W^{-\frac{1}{2}}w^2}\br{1+\sum_{r=1}^\infty n^{-\frac{r}{4}} p_r(w)} dw,
\]
where each $p_r(w)\in\CC[w]$ arises from multiplying the appropriate bracketed terms in \eqref{eq:cir4} and \eqref{eq:cir5}.

As $n\rightarrow\infty$, we see that the above expression is asymptotic to 
\begin{align*}
\frac{(-1)^{n}\gamma}{2\sqrt{n}\pi} e^{2\sqrt{nW}} W^{-\frac{1}{4}} \int_{-\infty}^\infty e^{-W^{-\frac{1}{2}}w^2}\br{1+\sum_{r=1}^\infty n^{-\frac{r}{4}} p_r(w)} dw
= \frac{(-1)^{n}\gamma}{2 \sqrt{n\pi}} e^{2\sqrt{nW}} \br{1+O\br{n^{-\frac{1}{2}}}}.
\end{align*}
The polynomial terms in \eqref{eq:cir4} and \eqref{eq:cir5} are odd functions of $w$ when $r$ is odd; thus, the corresponding terms in the above integral cancel and the integral of the second term is $O\left(n^{-\tfrac{1}{2}}\right)$.

Analogously, we obtain that as $n \to \infty$ the contribution from the second term in \eqref{eq:cirasy} is 
\[
\frac{(-1)^{n}\overline{\gamma}}{2 \sqrt{n\pi}} e^{2\sqrt{n\overline{W}}} \br{1+O\br{n^{-\frac{1}{2}}}}.
\]
A simple calculation shows that the final term in \eqref{eq:cirasy} contributes
\[
O\br{\frac{e^{\sqrt{n}\,\re{\sqrt{W}}}}{\sqrt{n}}}.
\]
Putting everything together, we obtain that as $n \to \infty$
\[
M(n)=\br{\frac{(-1)^{n}\gamma}{2 \sqrt{n\pi}} e^{2\sqrt{nW}} + \frac{(-1)^{n}\overline{\gamma}}{2 \sqrt{n\pi}} e^{2\sqrt{n \overline{W}}}} \br{1+O\br{n^{-\frac{1}{2}}}} + O\br{\frac{e^{\sqrt{n}\,\re{\sqrt{W}}}}{\sqrt{n}}}.
\]
This expression can be written as
\begin{equation*}
\begin{split}
M(n)&=(-1)^{n}\frac{e^{2\sqrt{n}\Re(\sqrt{W})}}{\sqrt{n\pi}}\br{\Re(\gamma)\cos(2\sqrt{n}\Im(\sqrt{W}))-\Im(\gamma)\sin(2\sqrt{n}\Im(\sqrt{W}))}\br{1+O\br{n^{-\frac{1}{2}}}}\\ 
&\qquad \qquad \qquad \qquad \qquad + O\br{n^{-\tfrac 1 2} e^{\sqrt{n}\,\re{\sqrt{W}}}}\\ 
&=(-1)^{n}\abs{\gamma}\frac{e^{2\sqrt{n}\Re(\sqrt{W})}}{\sqrt{n \pi}}\cos(2\sqrt{n}\Im(\sqrt{W})+\arg(\gamma))\br{1+O\br{n^{-\frac{1}{2}}}} + O\br{\frac{e^{\sqrt{n}\,\re{\sqrt{W}}}}{\sqrt{n}}}.
\end{split}
\end{equation*}
This gives the contribution from the major arc.

We next estimate a bound for the second integral in \eqref{eq: split} along the minor arc.

\subsection*{Minor arc estimate} 
The associated error term in \eqref{eq: split} is given by
\[
E(n)\coloneq\frac{1}{2\pi i}\int_{C\setminus C'}\frac{v(q)}{q^{n+1}} dq.
\]

Let $\zeta$ be a root of unity of order $m\neq2$, where $m\equiv 2\mod 4$.
Using the asymptotics from \eqref{eq: asymp v} and then applying the saddle point method as in the analysis of the major arc contribution, we find that the contribution of $\zeta$ can be written as a sum of two integrals of the form 
\[
\frac{\zeta^{-n}K{(\zeta)}}{2n^{1/4}\pi i}\int_{\Gamma_d} e^{\sqrt{n}\br{\frac{V}{d^2 z}+z}}z^{-\frac{1}{2}}dz,
\]
together with an error term of size 
\[
O\br{\frac{e^{\sqrt{n}\,\re{\sqrt{W}}}}{\sqrt{n}}},
\]
where $K(\zeta) = \gamma(\zeta) \text{ and } V=W$, or $K(\zeta) = \overline{\gamma}(\zeta) \text{ and } V=\overline{W}$.
Here $\Gamma_d$ is a contour passing through the corresponding saddle point, as described in the previous section.
As before, these integrals have the asymptotic contributions 
\[
\frac{\zeta^{-n}K{(\zeta)}}{2\sqrt{n\pi}} e^{\tfrac 2 d\sqrt{nV}} \br{1+O\br{n^{-\frac{1}{2}}}} \text{ as } n \to \infty.
\]

The dominant contribution comes from the $6$-th ordered roots of unity.
Therefore, as explained in the outline of the proof, we bound the entire error term $E(n)$ by the contribution from the $6$-th ordered roots multiplied by $2\pi$.
This gives
\[
E(n) = O\br{\frac{e^{\frac{2}{3}\sqrt{n}\,\re{\sqrt{W}}}}{\sqrt{n}}} + O\br{\frac{e^{\sqrt{n}\,\re{\sqrt{W}}}}{\sqrt{n}}}
= O\br{\frac{e^{\sqrt{n}\,\re{\sqrt{W}}}}{\sqrt{n}}},
\]
which completes the proof of the theorem.
\end{proof}

\subsection{Special cases of the main result}
We now restrict Theorem~\ref{Fourier: mainasym} to members of the family of $q$-hypergeometric series $\{\VVV_{r,b}(q): {(r,b)\in\{1,2\}\times \Z}\}$ defined in \eqref{def: qgeneral}.
This allows us to recover Theorem~\ref{thm:B} of the introduction; see Section~\ref{intro}.
\begin{corollary}
\label{cor: asympmain}
Fix $r\in\{1,2\}$ and $b\in\Z$.
Let $W_{2r}$ be as defined in \eqref{eq:Wdef}.
As $n \to \infty$,
\begin{align*}
V_{r,b}(n) 
&=(-1)^{n}\frac{e^{2\sqrt{n}\Re(\sqrt{W_{2r}})}}{\sqrt{2rn\sqrt{3}}}\cos\br{2\sqrt{n}\Im(\sqrt{W_{2r}})
+\frac{2\*r-1-2\*b}{12} \pi
}\br{1+O\br{n^{-\frac{1}{2}}}}\\
& \;+ O\br{\frac{ e^{\sqrt{n}\,\re(\sqrt{W_{2r}})}}{\sqrt{n}}}.
\end{align*}
\end{corollary}

\begin{proof}
The asymptotics for $V_{r,b}(n)$ follow directly from Theorems~\ref{thm:asympv2O1},~\ref{thm:asympv2}, and~\ref{Fourier: mainasym}, upon substituting the corresponding values of $W=W_{2r}$ from \eqref{eq:Wdef} and $\gamma = \gamma_{r,b}(-1)$ from \eqref{defgammazeta1}.
\end{proof}

Taking $r=1, \; b=-1$ and $r=2, \; b=0$ in Corollary~\ref{cor: asympmain}, we obtain the asymptotics as $n \to \infty$ for $V_2(n)$ and $V_4(n)$ respectively, 
\begin{align*}
 V_2(n) = (-1)^{n}\frac{e^{2\sqrt{n}\Re(\sqrt{W_{2}})}}{\sqrt{2n\sqrt{3}}}&\cos\br{2\sqrt{n}\Im(\sqrt{W_{2}}) + \frac{\pi}{4}}\br{1+O\br{n^{-\frac{1}{2}}}} + O\br{\frac{ e^{\sqrt{n}\,\re(\sqrt{W_{2}})}}{\sqrt{n}}},\\
 V_4(n) = (-1)^{n}\frac{e^{2\sqrt{n}\Re(\sqrt{W_{4}})}}{2\sqrt{n\sqrt{3}}}&\cos\br{2\sqrt{n}\Im(\sqrt{W_{4}}) + \frac{\pi}{4}}\br{1+O\br{n^{-\frac{1}{2}}}} + O\br{\frac{ e^{\sqrt{n}\,\re(\sqrt{W_{4}})}}{\sqrt{n}}}.
\end{align*}
The relation between $v_1$, $v_3$, and $v_4$ given by \eqref{eq:pointed by KB} yields
\[
V_3(n) = -(-1)^{n}\frac{e^{2\sqrt{n}\Re(\sqrt{W_{4}})}}{\sqrt{n\sqrt{3}}}\cos\br{2\sqrt{n}\Im(\sqrt{W_{4}}) + \frac{\pi}{4}}\br{1+O\br{n^{-\frac{1}{2}}}} + O\br{\frac{ e^{\sqrt{n}\,\re(\sqrt{W_{4}})}}{\sqrt{n}}},
\]
as $n \to \infty$.
This follows as the special case $b=0$ of the following more general corollary.

\begin{corollary}
\label{cor: asympU}
Fix $b\in\Z$.
Let $\kappa_1$ and $\kappa_2$ be non-zero integers.
Define
\begin{equation}
\label{def: U}
\UUU_{2,b}(q) \coloneqq \kappa_1 v_1(\pm q) + \kappa_2\VVV_{2,b}(q) \eqqcolon \sum_{n\ge0} U_{2,b}(n) q^n.
\end{equation}
As $n \to \infty$,
\begin{align*}
U_{2,b}(n) = (-1)^{n}\kappa_2\frac{e^{2\sqrt{n}\Re(\sqrt{W_{4}})}}{2\sqrt{n\sqrt{3}}}&\cos\br{2\sqrt{n}\Im(\sqrt{W_{4}}) +\frac{3-2b}{12}\pi
}\br{1+O\br{n^{-\frac{1}{2}}}} + O\br{\frac{ e^{\sqrt{n}\,\re(\sqrt{W_4})}}{\sqrt{n}}}.
\end{align*}
\end{corollary}

\begin{proof}
As $n\to\infty$, we see from \cite[Theorem~1.3]{folsom2023oscillating} that $V_1(n)$ is absorbed into the error term $O\br{\frac{ e^{\sqrt{n}\,\re({\sqrt{W_4}})}}{\sqrt{n}}}$ of $V_{2,b}(n)$.
Hence, the result follows from Corollary~\ref{cor: asympmain}.
\end{proof}

\section{The Sign Patterns of \texorpdfstring{$V_j(n)$}{}}
\label{sec:pfs}

We now prove a general result that shows that the asymptotic expansion for the coefficients $V_j(n)$ in Theorem~\ref{thm:B} implies the almost alternating sign pattern explained in Theorem~\ref{thm:sgnpattern}.
The result we prove in Theorem~\ref{thm:gen} essentially shows that when a sequence of integers $V(n)$ satisfies the asymptotic formula of the type appearing in Theorem~\ref{Fourier: mainasym}, it exhibits an almost alternating sign pattern.
In particular, this sign pattern holds for the coefficients $V_{r,b}(n)$ and $U_{2,b}(n)$ of the general $q$-series $\VVV_{r,b}(q)$ and $\UUU_{2,b}(q)$ defined in \eqref{def: qgeneral} and \eqref{def: U}, respectively.

To prove the main result in this section, we need a technical lemma which provides a quantitative measure of how far a given sequence is from being equidistributed.

\begin{lemma}
\label{lemma: schoissegeier}
Given a sequence of real numbers $(x_n)$ in $[0,1)$, define \textit{discrepancy} $D_N((x_n))$ by
\[
D_N((x_n)) \coloneq \sup_{0\le c\le d\le 1}\abs{\frac{\abs{\cbr{x_1,\dots,x_N}\cap [c,d)}}{N}-(d-c)}.
\]
If $R>0$, then $(x_n) \coloneq (R\sqrt{n})\mod 1$ has discrepancy 
\[
D_N((x_n))\ll N^{-1/2},
\]
as $N\to\infty$.
\end{lemma}

\begin{proof}
From \cite[Corollary 6]{Schoss}, in the case $R^2 \notin \mathbb{Q}$ we have
\[
\lim_{N\to\infty} \sqrt{N}\, D_N((x_n)) = \frac{1}{4R}.
\]
In particular,
\[
D_N((x_n)) \ll N^{-1/2}.
\]
\footnote{To derive \cite[(44)]{folsom2023oscillating}, a reference was made to the work of Schoi{\ss}engeier which only addresses the case $R^2\not\in \Q$.
The work of Baxa mentioned here, is also needed to obtain the desired inequality in \textit{loc. cit.}}In the case $R^2 \in \mathbb{Q}$, \cite[Theorem~4]{baxa1998discrepancy} implies that
\[
\limsup_{N\to\infty} \sqrt{N}\, D_N((x_n)) < \infty,
\]
and hence
\[
D_N((x_n)) \ll N^{-1/2}. \qedhere
\]
\end{proof}

The following variant with a fixed shift of the sequence $(R\sqrt{n})$ is a consequence of Lemma~\ref{lemma: schoissegeier}.

\begin{lemma}
\label{lemma: schoissegeier mod}
Fix $R>0$ and a real number $\beta$.
The sequence $(y_n) \coloneq (R\sqrt{n} + \beta)\mod 1$ has discrepancy 
\[
D_N((y_n))\ll N^{-1/2},
\]
as $N\to\infty$.
\end{lemma}

\begin{proof}
Let $(x_n) \coloneq (R\sqrt{n})$.
If
\[
y_n\coloneq (x_n+\beta)\mod1,
\]
then for any interval $I\subset[0,1)$,
\[
y_n\in I
\quad\Longleftrightarrow\quad
x_n\mod1\in (I-\beta)\mod 1.
\]
The set $(I-\beta)\mod{1}$ is either a single interval in $[0,1)$ of length $\abs{I}$ or the union of two disjoint intervals $I_1, I_2 \subset [0,1)$, in which case
\[
\abs{I_1} + \abs{I_2} = \abs{I}.
\]
Hence 
\[
\#\{n\le N:y_n\in I\}
=
\#\{n\le N:x_n\mod 1\in I_1\} + \#\{n\le N:x_n\mod 1\in I_2\},
\]
where we set $I_1 = I$ and $I_2=\emptyset$ if $(I-\beta)\mod1$ is a single interval.
Therefore,
\[
\abs{\frac{\#\{n \le N : y_n \in I\}}{N} - \abs{I}}\le 2 D_N((x_n\mod 1)).
\]
Taking the supremum over all such $I\subset [0,1)$ yields
\[
D_N((y_n)) \le 2 D_N((x_n\mod1)).
\]
Using Lemma \ref{lemma: schoissegeier}, we thus have 
\[
D_N((y_n)) \ll N^{-1/2}. \qedhere
\]
\end{proof}

We now establish a lemma controlling the number of points lying in shrinking intervals, under the above discrepancy bound.

\begin{lemma}
\label{lem: Var interval}
Let $(x_n)$ be a sequence in $[0,1)$ with discrepancy
\[
D_N((x_n)) \ll N^{-1/2}.
\]
Fix $\vartheta_1, \vartheta_2 \in [0,1)$.
For $j=1,2$ and each $n$, define intervals
\[
I_{j,n} := \left[\vartheta_j - \frac{C}{\sqrt{n}},\, \vartheta_j + \frac{C}{\sqrt{n}}\right]\mod{1},
\]
for some constant $C>0$.
Then \[
A_N := \#\{n \le N : x_n \in I_{1,n} \cup I_{2,n}\} \ll N^{1/2}.
\]
In particular, $A_N/N \to 0$ as $N \to \infty$.
\end{lemma}

\begin{proof}
We partition the range $1 \le n \le N$ into dyadic intervals $2^k \le n < 2^{k+1}$.
Let us use the standard notation $A \asymp B$ to mean that there exist constants $k, K > 0$ such that $k\abs{B} \le \abs{A} \le K\abs{B}$.
Then in a dyadic interval, we have $n^{-1/2} \asymp 2^{-k/2}$ and
\[
I_{j,n} \subset I_{j,k},
\]
where $I_{j,k}$ is an interval in $[0,1)$ centered at $\vartheta_j$ of radius $2^{-k/2}$.
Thus,
\[
A_N \le \sum_{k \le \log_2 N} \#\{2^k \le n < 2^{k+1} : x_n \in I_{1,k} \cup I_{2,k}\}.
\]

By the discrepancy bound, for any interval $I \subset [0,1)$ we have
\[
\#\{n \le M : x_n \in I\} \ll M\abs{I} + M^{1/2}.
\]
Applying this with $M = 2^{k+1}$ and $I = I_{j,k}$ yields
\[
\#\{2^k \le n < 2^{k+1} : x_n \in I_{j,k}\} \le \#\{n \le 2^{k+1} : x_n \in I_{j,k}\} \ll 2^{k/2}.
\]

Summing over $k$ and $j=1,2$, we obtain
\[
A_N \ll \sum_{k \le \log_2 N} 2^{k/2} \ll N^{1/2}. \qedhere
\]
\end{proof}

We now prove the main result of this section.

\begin{theorem}
\label{thm:gen}
Let $V(n)$ be a sequence of integers satisfying
\[
V(n) \= (-1)^{n}\alpha\frac{e^{2\sqrt{n}\Re(\sqrt{W})}}{\sqrt{n}}
\cos\bigl(2\sqrt{n}\Im\bigl(\sqrt{W}\bigr)+ \beta\bigr)
\left(1+O(n^{-\frac{1}{2}})\right) + O\left(\frac{e^{\sqrt{n}\Re(\sqrt{W})}}{\sqrt{n}}\right)
\]
as $n\to \infty$, where $\alpha\neq 0$, $\beta$ are real numbers and $W$ is a complex number with $0 \le \arg W < \pi$.
If $\Im(\sqrt{W})\neq 0$, then $V(n)$ satisfies an almost alternating sign pattern in the sense of Theorem~\ref{thm:sgnpattern}.
In particular, for $r \in \{1,2\}$ and $b\in \Z$, the coefficient sequences $V_{r,b}(n)$ and $U_{2,b}(n)$ satisfy an almost alternating sign pattern.
\end{theorem}

\begin{proof}
Let $M^*(n)$ and $E(n)$ be any sequences of real numbers such that 
\begin{align*}
& M^*(n) =(-1)^n\alpha\frac{e^{2\sqrt{n}\Re(\sqrt{W})}}{\sqrt{n}}
\cos\bigl(2\sqrt{n}\Im(\sqrt{W}) + \beta\bigr) \left(1+O(n^{-\frac{1}{2}})\right),\\
& E(n) =O\!\left(\frac{e^{\sqrt{n}\Re(\sqrt{W})}}{\sqrt{n}}\right).
\end{align*}
Let $x_n := 2\sqrt{n}\Im(\sqrt{W}) + \beta$.
We first show that
$\sign(V(n)) = \sign(M^*(n))$ whenever
\[
\abs{\cos(x_n)} \gg \frac{1}{\sqrt{n}}.
\]
Indeed, in this case
\[
\abs{M^*(n)} \sim \abs{\alpha}\frac{e^{2\kappa\sqrt{n}}}{\sqrt{n}}\abs{\cos(x_n)}
\gg \frac{e^{2\kappa\sqrt{n}}}{n},
\]
where $\kappa := \Re(\sqrt{W})$ and we have
\[
\abs{E(n)} \ll \frac{e^{2\kappa\sqrt{n}}}{n},
\]
so that $\abs{M^*(n)} \gg \abs{E(n)}$, which implies that $V(n)$ and $M^*(n)$ have the same sign.

\medskip

\noindent \textit{Claim:} The condition $\abs{\cos(x_n)}\gg\frac{1}{\sqrt{n}}$ holds for almost all $n$.

\medskip

\noindent \textit{Proof of claim:} Let $\vartheta_1,\vartheta_2 \in [0,2\pi)$ be the zeros of $\cos(x)$.
By Taylor expansion, for $j=1,2$ and $x$ in a neighbourhood of $\vartheta_j$, 
\[
\cos(x) = \pm (x - \vartheta_j) + O\bigl((x - \vartheta_j)^2\bigr),
\]
which implies $\abs{\cos(x)} \sim \abs{x - \vartheta_j}$.

Define $x_n'\coloneq \frac{x_n}{2\pi}\mod1$, $\vartheta_j'\coloneq \frac{\vartheta_j}{2\pi}$ and 
\[
\mathcal{E} := \left\{ n \in \mathbb{N} : \abs{x_n' - \vartheta_j'} < Cn^{-1/2} \text{ for some } j \in \{1,2\} \right\},
\]
for some $C>0$.

Applying Lemmas~\ref{lemma: schoissegeier mod} and~\ref{lem: Var interval} to the sequence $(x_n')$, it follows
\[
\frac{\#\{n \le N : n \in \mathcal{E}\}}{N} \to 0,
\]
so that $\mathcal{E}$ has natural density zero.
In particular, this proves the claim.

\medskip

Also note that for any real number $R$,  by the mean value theorem
\[
\cos(R\sqrt{n+1}+\beta) - \cos(R\sqrt{n}+\beta) = O(n^{-1/2}).
\]
So, whenever $\abs{\cos(R\sqrt{n} + \beta)}$ exceeds a sufficiently large constant multiple of $n^{-1/2}$, the sign of $\cos(R\sqrt{n} + \beta)$ is stable under increments of $n$.

Putting everything together, we conclude that the sign pattern of $M^*(n)$ alternates with parity outside a set of density zero, and therefore $V(n)$ and $V(n+1)$ have alternate signs for almost all $n$.

Moreover, for $n \notin \mathcal{E}$,
\[
\abs{V(n)} \ge \abs{M^*(n)} - \abs{E(n)} \gg \frac{e^{2\kappa\sqrt{n}}}{n}.
\]
Since $e^{2\kappa\sqrt{n}}/n \to \infty$, it follows that $\abs{V(n)} \to \infty$ away from a set of density zero.

To prove the statement about the specific coefficient sequence $V_{r,b}(n)$, we refer to Corollary~\ref{cor: asympmain} and substitute
\[
\alpha = \frac{1}{\sqrt{2r\sqrt{3}}}, 
\qquad
\beta = (-1)^{r+1}\frac{2(b+2)+r}{12r}\pi,
\qquad
W = W_{2r}.
\]
The case of $U_{2,b}(n)$ follows similarly from Corollary~\ref{cor: asympU} and the main assertion of this theorem.
\end{proof}

\begin{remark}
\label{rem: degenerate}
If $\Im(\sqrt{W}) = 0$ and $\beta$ is not an odd multiple of $\pi/2$, then the cosine term is constant and non-zero, so $V(n)$ has strictly alternating sign for all large $n$.
\end{remark}

The following corollary completes the proof of Theorem~\ref{thm:sgnpattern}.

\begin{corollary}
The sequence of coefficients $V_j(n)$ for $j\in \{2,3,4\}$ satisfies an almost alternating sign pattern in the sense of Theorem~\ref{thm:sgnpattern}.
\end{corollary}

\begin{proof}
Take  $(r,b)$ to be $(1,-1)$ and $(2,0)$ for $V_{r,b}(n)$, and $b=0$ for $U_{2,b}(n)$.
This yields the sign patterns of $V_2(n)$, $V_4(n)$, and $V_3(n)$, respectively.
\end{proof}

\section{Future directions and further examples}
\label{sec:future}

We provide additional examples and a new family of $q$-series whose coefficients appear to behave in a way similar to the coefficients $V_j(n)$.
In all these examples, the behaviour can presumably be explained as follows.
As $q$ approaches certain roots of unity of order greater than 1, the asymptotics of the $q$-series become large.
Following the circle method, this leads to exponential growth and oscillation (possibly degenerate as in Remark~\ref{rem: degenerate}) in the coefficients. 
By Theorem~\ref{thm:gen} (or a suitable generalization), the coefficients \textit{should} satisfy an almost (alternating) sign pattern.

In recent times, $q$-series with similar patterns have been noticed; see for example, \cite[Section~6]{KKJparbias}.
Moreover, the function $v_3(q)$ also appears in \cite[Equation (3.1)]{AGARWAL1984291}.
After making the substitution
$\alpha=\frac{q}{\tau}$, $\beta=\delta=-q$, $t=\tau$, $\gamma\to 0,$ and $\tau\to 0$,
we have
\begin{equation}
\label{eq:pointed by KB2}
v_3(q) 
\= \frac1{(-q;q)_\infty} v_5(q) w(q)
+\frac2{(-q;q)_\infty}
\sum_{n\geq 0}\frac{q^{n(n+1)/2}}{(q;q)_n}
\sum_{m\geq 0} \frac{(-1)^mq^{m(m+1)/2+nm}}{(q^{n+1};q)_m(-q^n;q)_m},
\end{equation}
where, using \cite[Entry~9.4.1]{RlostI},
\[
v_5(q)\;\coloneq\; \sum_{n\geq 0}(-1)^n\frac{q^{n(n+1)/2}}{\br{q^2;q^2}_n},\qquad\quad
 w(q)\;\coloneq\;
 1 - \sum_{k>0} q^{(3k^2 - k) / 2} (1 - q^k).
\]
The function $v_5(q)$ can also be found in \cite[p.10]{Rlost} and its coefficients seem to satisfy a similar sign pattern as described in Theorem~\ref{thm:sgnpattern}.

We now introduce an infinite family of $q$-series $\{v^{\{k\}}(q)\}$ whose coefficients \textit{appear} to exhibit an almost alternating sign pattern of length $k$.
For $k\in\Z_{\geq2}$, we define 
\[
v^{\{k\}}(q)
\coloneqq 
\br{q,q}_\infty\sum_{n\geq 1}\frac{q^{n(n-1)}}{\br{-q;q^k}_n}
\eqqcolon
\sum_{n\geq 0} V^{\{k\}}(n)q^n.
\]

We notice that when $k=2$, the coefficients $V^{\{2\}}(n)$ behave similar to the coefficients $S(n)$ of $\sigma$ as in~\eqref{conjS}.
This is illustrated in Figures \ref{fig:S} and \ref{fig:vk2}, where the colours distinguish odd and even values of $n$.
The sequence $V^{\{2\}}(n)$ appears to take all integer values infinitely often and to satisfy $\limsup_{n\to\infty} \abs{V^{(2)}(n)} = \infty$, as is known for $\sigma$~\cite{andrews1988partitions}.

\begin{figure}[H]
\centering
\begin{minipage}{.45\textwidth}
\centering
\includegraphics[width=1\linewidth]{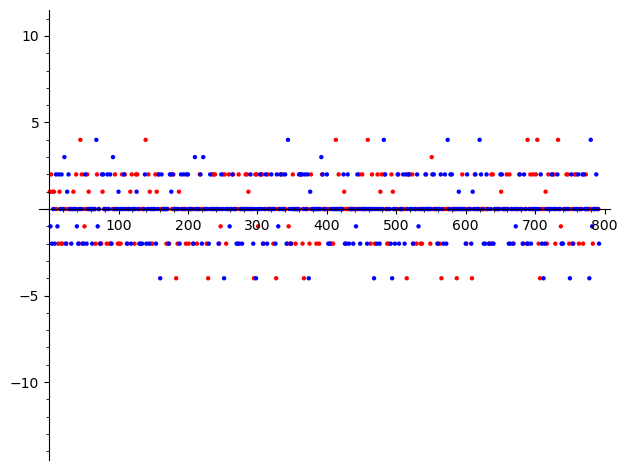}
\caption{Coefficients $S(n)$ for $0\le n\le 800$.}
\label{fig:S}
\end{minipage}
\begin{minipage}{.45\textwidth}
\centering
\includegraphics[width=1\linewidth]{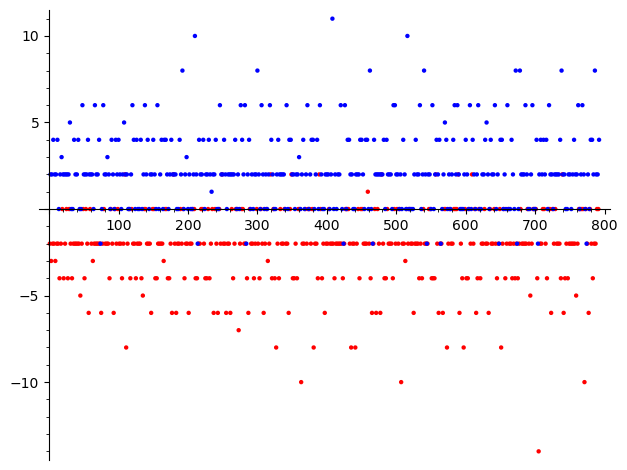}
\caption{Coefficients $V^{\{2\}}(n)$ for $0\le n\le 800$.}
\label{fig:vk2}
\end{minipage}
\end{figure}

The distinct colours also indicate an almost alternating sign pattern for the coefficients $V^{\{2\}}(n)$ in the sense of Theorem~\ref{thm:sgnpattern}.
Moreover, the $q$-series satisfies the following relation with the Ramanujan's third order mock theta function $\nu$
\[
v^{\{2\}}(q) = (q;q)_\infty \; \nu(q),
\]
which follows from the definition of both the $q$-series.
Further it was shown in \cite[Equation~(2.6)]{mortenson2013three} that $\nu(q)$ has the following Hecke-type series representation.
\begin{align*}
\nu(q) \coloneqq \sum_{n\geq 0}\frac{q^{n(n+1)}}{\br{-q,q^2}_{n+1}}  &= \frac{(-q^2;q^2)_\infty}{(q^2;q^2)_\infty}\sum_{n\geq 0} (1-q^{2n+1})q^{3n^2 + 2n} \sum_{j=-n}^n (-1)^j q^{-j^2}\\
& = \frac{1}{(-q; -q)_{\infty}}\sum_{n\geq 0} \sum_{j={-n}}^n (-1)^{n-j(j+1)/2} q^{2n^2 + 2n - j(j+1)/2}.
\end{align*}

We next consider the example for $k=3$ to explain the idea of almost alternating sign pattern of length $k$.
We have
\[
v^{\{3\}}(q) \=1
 - 2\*q
 + 2\*q^2
 - 3\*q^3
 + 2\*q^4
 -q^5
 +q^6
 +q^7
 -q^8
 +2\*q^9
 -2\*q^{10}
 +2\*q^{11}
 +\cdots
\]
and the coefficients $V^{\{3\}}(n)$ for $0\le n\le200$ are plotted in Figure~\ref{fig:coeffV23}.
Each colour tracks a fixed congruence class modulo $3$.
We observe that the coefficients follow a very structured sign pattern.
For almost all $n$, three consecutive coefficients $V^{\{3\}}(n), V^{\{3\}}(n+1), V^{\{3\}}(n+2)$ contain two positive and one negative number or vice versa.

\begin{figure}
\centering
\includegraphics[width=0.6\linewidth]{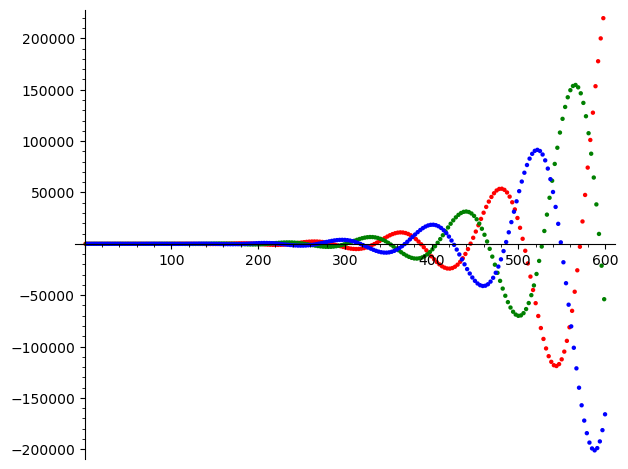}
\caption{The coefficients $V^{\{3\}}(n)$ for $0\leq n\leq 200$.}
\label{fig:coeffV23}
\end{figure}

More generally, we see similar patterns for the other values of $k$.
Our observations lead us to the following question -- an answer to which will provide improved understanding of (some) $q$-series.

\begin{ques}
We ask whether, for almost all $n$, the sequence $V^{\{k\}}(n), V^{\{k\}}(n+1),\dots, V^{\{k\}}(n+k-1)$ of $k$ consecutive coefficients satisfies the following sign pattern: 
\begin{enumerate}[label=(\alph*)]
\item if $k$ is even, it contains an equal number of positive and negative terms.
\item if $k$ is odd, it contains either $\lfloor\frac{k}2\rfloor$ positive and $\lceil\frac{k}2\rceil$ negative terms or vice-versa.
\end{enumerate}
\end{ques}

More precisely, the asymptotics of the coefficients $V^{\{k\}}(n)$ appear to depend on $n \mod{k}$.
Based on numerical experiments, for each $n_0=0,\ldots,k-1$, the asymptotics of the subsequence $V^{\{k\}}(n_0+\ell k)$ with $\ell \geq 0$ should have an oscillating and an exponential factor, which would explain the sign pattern in the previous question.

\bibliographystyle{abbrv}
\bibliography{refs.bib}

\end{document}